\documentclass[12pt]{amsart}
\usepackage[matrix,cmtip,arrow,curve]{xy}
\usepackage[unicode]{hyperref}
\usepackage{graphicx}
\usepackage{color}
\usepackage{amssymb}
\usepackage{enumerate}

\makeatletter
\@namedef{subjclassname@2010}{%
  \textup{2010} Mathematics Subject Classification}
\makeatother

\numberwithin{equation}{section}

\theoremstyle{plain}
\newtheorem{thm}[equation]{Theorem}
\newtheorem{lem}[equation]{Lemma}

\newtheorem{prop}[equation]{Proposition}
\theoremstyle{definition}
\newtheorem{defn}[equation]{Definition}
\newtheorem{rem}[equation]{Remark}
\newtheorem{example}[equation]{Example}

\newcommand{\Hom}{\operatorname{Hom}}

\newcommand{\ob}{\operatorname{Ob}}
\newcommand{\mor}{\operatorname{Mor}}
\newcommand{\mcc}{\mathcal{C}}

\newcommand{\id}{\operatorname{id}}
\newcommand{\set}[1]{\{#1\}}
\newcommand{\setm}[2]{\{\,#1\mid#2\,\}}
\newcommand{\ci}{\subseteq}
\newcommand{\inv}{{-1}}

\newcommand{\col}{\operatorname{Col}}

\newcommand{\Set}{{\mathcal{S}et}}

\newcommand{\nerve}{\operatorname{nerve}}

\newcommand{\map}{\operatorname{Map}}

\newcommand{\dcop}{\Delta \circlearrowright \Omega_p}
\newcommand{\dcos}{\Delta \circlearrowright \Omega}
\newcommand{\operads}{\mathcal{O}perad}
\newcommand{\nsoperads}{\operads_{ns}}
\newcommand{\Iso}{\operatorname{Iso}}
\newcommand{\Ob}{\operatorname{Ob}}
\newcommand{\pullback}[2]{\,_{#1}\!\times_{#2}}
\newcommand{\cat}{\mathcal{C}at}

\newcommand{\setact}{{\mathcal{A}ct\mathcal{S}et}}
\newcommand{\catact}{{\mathcal{RA}}}
\newcommand{\opact}{{{\mathcal{RA}^{\mathcal{O}p}}}}
\newcommand{\opactns}{{{\mathcal{RA}^{\mathcal{O}p}_{ns}}}}

\newcommand{\mcd}{\mathcal{D}}
\newcommand{\mca}{\mathcal{A}}
\newcommand{\mcb}{\mathcal{B}}
\newcommand{\mco}{\mathcal{O}}
\newcommand{\mcp}{\mathcal{P}}

\newcommand{\circlerightarrow}{\circlearrowright}
\newcommand{\betweener}{\overset\bullet\circlearrowright}
\newcommand{\bi}[2]{\ensuremath{  \left[ #1 \circlearrowright #2       \right]      }}

\title[Categories encoding category actions]{Reedy categories which encode the notion of category actions}

\author[J. E. Bergner]{Julia E. Bergner}
\address{Department of Mathematics\\ University of California, Riverside\\ Riverside, CA 92521}
\email{bergnerj@member.ams.org}
\author[P. Hackney]{Philip Hackney}
\email{hackney@math.ucr.edu}

\date{}



\thanks{The first-named author was partially supported by NSF grants DMS-0805951 and DMS-1105766, and by a UCR Regents Fellowship.}

\begin{document}

%%%%% To ease editing, for IMPAN journals add:

\baselineskip=17pt

\begin{abstract}
We study a certain type of action of categories on categories and on operads.
Using the structure of the categories $\Delta$ and $\Omega$ governing category and operad structures, respectively, we define categories which instead encode the structure of a category acting on a category, or a category acting on an operad.  We prove that the former has the structure of an elegant Reedy category, whereas the latter has the structure of a generalized Reedy category.  In particular, this approach gives a new way to regard group actions on categories and on operads.
\end{abstract}
\maketitle

\section{Introduction}

The simplicial category $\Delta$ can be found in numerous contexts, in homotopy theory, category theory, and beyond.  In some sense, the structure of $\Delta$ indexes the structure of a category: $[0]$ indexes objects, $[1]$ indexes morphisms, $[2]$ indexes composition of morphisms, and so forth.  One picks out this structure in a category by taking the nerve functor, resulting in a simplicial set.  Many of the models for $(\infty, 1)$-categories, such as Segal categories, complete Segal spaces, and quasi-categories, are given by simplicial objects of some kind and therefore make use of this formalism to make sense of categories up to homotopy.  While not all simplicial diagrams give a category structure, the Segal condition allows us to identify those that do, either strictly or up to homotopy.  Using a modified version of the Segal condition, first introduced by Bousfield \cite{bousfield}, the category $\Delta$ also governs groupoid structures, and in particular the special case of the structure of a group.

Much more recently, Moerdijk and Weiss have introduced the dendroidal category $\Omega$ which plays the same role for the structure of a colored operad \cite{mw}.  The objects are finite rooted trees, specifying every kind of composition that can take place.  Hence, they were able to understand $(\infty, 1)$-operads as dendroidal diagrams, and Cisinski and Moerdijk have successfully been able to compare many different models which arise in this way \cite{cm-ds}, \cite{cm-simpop}, \cite{cm-ho}.  Again, a Segal condition is necessary to understand which dendroidal objects can be regarded as some kind of colored operad.

The goal of the present paper is to give diagrams which govern group actions on categories and group actions on operads.
The particular type of action we study are called \emph{rooted actions}, see Section \ref{S:catoncat}.
Since the category $\Delta$ is used not only to encode groups, but categories more generally, we find a category $\Delta \circlearrowright \Delta$ encoding rooted actions of categories on categories and analogously $\Delta \circlearrowright \Omega$ encoding rooted actions of categories on colored operads. Restricting to the single-object case and imposing the Bousfield-Segal condition on the part of the diagram giving the acting category gives the special case of group actions.

Our motivation for this work arose in \cite{segaloperads}, in which we sought to give a proof of an alternative perspective on the Cisinski-Moerdijk results in the case of ordinary single-colored operads, making a comparison to simplicial operads regarded as algebras over the theory of operads.
However, we wished also to extend this result to have a comparison between the category of simplicial operads with a simplicial group action (where the acting group as well as the action can vary through the category) and some category of Segal-type diagrams over an appropriate category, namely $\Delta \circlearrowright \Omega$.

To motivate this construction, let us consider how we think of a group action on an operad; more details are given in \cite[\S 6]{segaloperads}.  An action of a group $G$ on an operad $P$ is simply an action of $G$ on $P(n)$ for each $n\geq 0$.  We do not insist upon any compatibility with the structure maps of $P$, so that we include the circle action on the framed little disks operad as an example.  As another example, suppose that $X$ is a $G$-space; then the endomorphism operad $\mathcal E_X$ has an action of $G$.

We begin with a method for encoding rooted actions of a category on another category, which can be restricted to the case of interest, where we have a group action. The category we obtain is denoted $\Delta \circlearrowright \Delta$.
We can extend to the diagram $\dcos$ which governs rooted actions of categories on operads, which answers our original question.  In \cite{segaloperads}, we establish the correct Segal condition to use in this framework and give an explicit Quillen equivalence between the corresponding model structure and the one on simplicial operads with a simplicial group action.

After defining these two diagrams of interest, we establish some properties they possess.  In the first case, we show that $\Delta \circlearrowright \Delta$ is a Reedy category, and in fact an elegant Reedy category in the sense of \cite{bergnerrezk}.  This property will be useful in future work in that it guarantees that, when we consider the category of functors from it to the category of simplicial sets, the Reedy and injective models are the same.  We also show elegance for a planar version of $\dcos$. The category $\dcos$ itself does not admit a Reedy structure, since objects may possess nontrivial automorphisms, but we show that it is a generalized Reedy category in the sense of \cite{bergermoerdijk}.

\subsection{The categories \texorpdfstring{$\Delta$}{\unichar{916}} and \texorpdfstring{$\Omega$}{\unichar{937}}}

The category $\Delta$ consists of the finite ordered sets $[n]=(0 \leq 1 \leq  \cdots  \leq n)$ and order-preserving maps between them.

The category $\Omega$, on the other hand, has as objects finite rooted trees.  For any such tree $T$, one can take the free colored operad on it, where each edge is assigned a distinct color; we denote this operad by $\Omega(T)$.  The morphisms $S \rightarrow T$ in $\Omega$ are defined to be the operad morphisms $\Omega(S) \rightarrow \Omega(T)$.

We also have the variation $\Omega_p$ whose objects are finite planar rooted trees.  While $\Omega$ governs symmetric colored operads, $\Omega_p$ governs nonsymmetric operads.  Further details about these categories and their relationship with operads can be found in \cite{moerdijklecture}.

\subsection{Reedy and generalized Reedy categories}

In this section, we briefly recall the definitions of Reedy category \cite{hirschhorn} and generalized Reedy category \cite{bergermoerdijk}. These two concepts provide a framework for working inductively in diagram categories. If $\mathcal{R}$ is a (generalized) Reedy category and $\mathcal{M}$ is any model category, then there is an associated model structure on the categories of diagrams $\mathcal{M}^{\mathcal{R}}$ \cite{bergermoerdijk}, \cite{hirschhorn}, \cite{reedy}.

A \emph{wide subcategory} of a category $\mathcal{C}$ is a subcategory which contains all objects of $\mathcal{C}$. A \emph{Reedy category} is a small category $\mathcal{R}$ together with two wide subcategories $\mathcal{R}^+$ and $\mathcal{R}^-$ and a degree function $d: \ob \mathcal{R} \to \mathbb{N}$ such that
\begin{itemize}
\item every non-identity morphism in $\mathcal{R}^+$ raises degree,
\item every non-identity morphism in $\mathcal{R}^-$ lowers degree, and
\item every morphism in $\mathcal{R}$ factors uniquely as a morphism in $\mathcal{R}^-$ followed by a morphism in $\mathcal{R}^+$.
\end{itemize}

In particular, Reedy categories cannot contain any non-identity automorphisms. The generalized Reedy categories of Berger and Moerdijk allow for such automorphisms.
We will write $\Iso(\mathcal{C})$ for the wide subcategory consisting of all isomorphisms in the category $\mathcal{C}$.
A \emph{generalized Reedy structure} on a small category $\mathcal{R}$ consists of
\begin{itemize}
\item
wide subcategories $\mathcal{R}^+$ and $\mathcal{R}^-$, and
\item
a degree function $d: \Ob(\mathcal{R}) \to \mathbb{N}$
\end{itemize}
satisfying the following four axioms.
\begin{enumerate}[(i)]
\item
Non-invertible morphisms in $\mathcal{R}^+$ (resp., $\mathcal{R}^-$) raise (resp., lower) the degree.  Isomorphisms in $\mathcal{R}$ preserve the degree.
\item
$\mathcal{R}^+ \cap \mathcal{R}^- = \Iso(\mathcal{R})$.
\item
Every morphism $f$ of $\mathcal{R}$ factors as $f = gh$ with $g  \in \mathcal{R}^+$ and $h \in \mathcal{R}^-$, and this
factorization is unique up to isomorphism.
\item
If $\theta f=f$ for $\theta \in \Iso(\mathcal{R})$ and $f\in \mathcal{R}^-$, then $\theta$ is an identity.
\end{enumerate} If, moreover, the condition
\begin{enumerate}[(iv')]
\item If $f \theta=f$ for $\theta \in \Iso(\mathcal{R})$ and $f\in \mathcal{R}^+$, then $\theta$ is an identity
\end{enumerate}
holds, then we call this a generalized \emph{dualizable}  Reedy structure.

\subsection{Categories acting on sets}
For a small category $\mcc$, we denote the set of objects by $C_0$ and the set of morphisms by $C_1$. We now define the notion of a category $\mcc$ acting on a set $A$ analogously to that of a groupoid acting on a set as found in \cite[\S 5.3]{mm}.
The data of such an action consists of
\begin{itemize}
\item a \emph{moment} map $\mu: A \to C_0$, and
\item an \emph{action} map $\bullet: C_1 \pullback{s}{\mu} A \to A$.
\end{itemize}
An action is required to satisfy the following axioms:
\begin{itemize}
\item the moment respects the action, in the sense that $\mu(f\bullet a) = t(f)$;
\item associativity, which is the usual action condition
\[
\xymatrix{
C_1 \pullback{s}{t} C_1 \pullback{s}{\mu} A \ar@{->}[r]_-{(\id, \bullet)} \ar@{->}[d]_{(\circ, \id)} & C_1 \pullback{s}{\mu} \ar@{->}[d]^\bullet A \\
C_1 \pullback{s}{\mu} A \ar@{->}[r]_-\bullet & A;
}\]
and,
\item the identity acts trivially: $\id_{\mu(a)} \bullet a = a$.
\end{itemize}

We write such an action as $\mcc \betweener A$.\footnote{We do not need to include $\mu$ in the notation since it can be recovered by examining $(C_0 \times A) \cap domain (\bullet) = (C_0 \times A) \cap (C_1 \pullback{s}{\mu} A) = \set{(\id_{\mu(a)},a)} \ci C_1 \times A$.}
The collection of all such actions forms a category $\setact$, where a morphism \[ X:\mcc \betweener A \to \mcc'
\overset\blacktriangledown\circlearrowright
A'\]
consists of a functor $X^c: \mcc \to \mcc'$ and a map of sets $X^s: A \to A'$ which satisfy $\mu X^s = X^c \mu$ and $X^s(f\bullet a) = X^c(f) \blacktriangledown X^s(a)$. In most situations we are working with a single action $\bullet$, and just write $\mcc \circlearrowright A$ for $\mcc \betweener A$.

\section{Rooted actions}\label{S:catoncat}

Consider two categories $\mcc$ and $\mcd$. An \emph{rooted action} of $\mcc$ on $\mcd$ is an action of the category $\mcc$ on the set $D_1$ of morphisms of $\mcd$, satisfying two additional axioms. We write $\mu: D_1 \to C_0$ for the moment map and $\bullet: C_1 \pullback{s}{\mu}D_1 \to D_1$ for the action map. The additional axioms are that
\begin{equation}
\mu(g\circ g') = \mu(g)
\end{equation}
for all $g$ and $g'\in D_1$ composable morphisms, and
\begin{equation}
s(f\bullet g) = s(g)
\end{equation}
for all $f\in C_1$, $g\in D_1$ with $s(f) = \mu(g)$.
We note that $\mu(\id_{t(g)}\circ g) = \mu(\id_{t(g)})$, so we could just as well define $\mu: D_0 \to C_0$.
The collection of all such actions forms a category $\catact$, where a morphism $X: \mcc \circlerightarrow \mcd \to \mcc' \circlerightarrow \mcd'$ is a pair of functors $X^c: \mcc \to \mcc'$ and $X^d: \mcd \to \mcd'$ which respect the moment and action maps.

\begin{rem}
A rooted action of $\mcc$ is distinct from other notions of actions.
For instance, in the standard notion of a groupoid acting on a groupoid, as found in \cite[\S 5.3]{mm}, the groupoid acts via functors.
If we consider a set as a category $\mcd$ with only identity morphisms, then there are no nontrivial rooted actions of $\mcc$ on $\mcd$, though there may many actions of $\mcc$ on $\Ob \mcd$.
\end{rem}

% \begin{rem}\label{R:settocat}
% One might be tempted to think of objects in $\setact$ as objects in $\catact$ where a set is considered as a category with only identity morphisms. It turns out that there are \emph{no} nontrivial actions on such a category.
% Let $Q$ be a set and $\mathcal{D}$ be the category with only identity morphisms and object set $Q$. Then no category $\mathcal{C}$ acts nontrivially on $\mathcal{D}$. Indeed, if we had $f\in C_1$ so that $f\bullet q = q' \neq q$, then we would also have $s(f\bullet q) = s(q') = q'$ and $s(q) = q$, which shows that $s(f\bullet q) \neq s(q)$.
% \end{rem}

A fundamental example is the following. We have the category $\mcc_{1,1}$ which is the free category on the diagram
\[ \xymatrix{
& 1  \\ \ast & 0\ar@{->}[u]_p
}\]
and $\mcd_{1,1}$,  the free category on
\[ \xymatrix{
y   & z \\
w \ar@{->}[r]^g \ar@{->}[u]^{p\bullet g} & x\ar@{->}[u]_{p\bullet \id_x}
}\]
with $\mu(w)=*$, $\mu(x) =0$, and $\mu(y) =\mu(z) = 1$. The action is as specified in the second diagram ($g$ and $\id_x$ are the only arrows that may be acted on by a non-identity element of $\mcc_{1,1}$ since they are the only arrows of moment $0$).

Suppose that $\mca$ and $\mcb$ are two other categories together with a rooted action of $
\mca$ on $
\mcb$. Then given any morphisms $a$ in
$\mca$ and $b$ in
$\mcb$ such that $\mu(b) = s(a)$, we obtain a morphism $X: \mcc_{1,1} \circlearrowright\mcd_{1,1} \to \mca \circlearrowright \mcb$ with $X^c(p) = a$ and $X^d(g) = b$. Here we have $X^d(p\bullet g) = a \bullet b$ and $X^d(p\bullet \id_x) = a\bullet \id_{t(b)}$, and $X^c(*) = \mu(s(b))$.
One sees that
\[ A_1 \pullback{s}{\mu} B_1 \cong \Hom_\catact(\mcc_{1,1} \circlearrowright\mcd_{1,1}, \mca \circlearrowright \mcb).\]
Thus this example is of supreme importance because it allows us to identify all  pairs of morphisms $(a,b)$ so that $a$ acts on $b$.

Our goal is to define a category which is the rooted action analogue of $\Delta$, in the sense that it allows us to form the ``nerve" of a rooted action $\mca \circlearrowright \mcb$, where an element of this nerve consists of a string of composable morphisms in $\mcb$ and a string of composable morphisms in $\mca$ such that we can act on the last morphism in the first list with the first morphism in the second list.
We first describe the objects $\bi{n}{k}$, which we think of
as a formal rooted action of $[n]$ on $[k]$.
The acting category, $\mcc_{n,k}$ has objects
\[ *_0, *_1, \dots, *_{k-1}, 0, 1, \dots, n \]
and is free with morphisms generated by $p_{i,i+1}:i\to i+1$; we write
\begin{equation}\label{pij}
p_{i,i+j}: i\to i+j
\end{equation}
for the unique map. The $*_\ell$ will merely serve as free targets for the moment map. There are no non-identity morphisms involving the $*_\ell$.

The category which is acted on, $\mcd_{n,k}$, is the empty category if $k=-1$.
Otherwise, it is a free category which is built inductively in $n$.
The base case is to define $\mcd_{0,k} = [k]$:
\[ \xymatrix{ 0 \ar@{->}[r] & 1 \ar@{->}[r] & 2 \ar[r] & \cdots \ar[r] & k-1 \ar@{->}[r] &k }\]
with
\[ \mu(i) = \begin{cases} *_i & 0 \leq i < k \\ 0 & i=k. \end{cases} \]
The category $\mcd_{1,k}$ is defined by adding a generating morphism for each morphism $h$ with target $0$, which are the formal actions of $p_{0,1}$ on $h$. We thus have a ladder shape for our generating graph
\[ \xymatrix{
\heartsuit & \heartsuit & \heartsuit &  \heartsuit &\heartsuit \\
0 \ar@{->}[r] \ar@{->}[u] & 1 \ar@{->}[u]\ar@{->}[r] & 2\ar@{->}[u] \ar@{.}[r] & k-1 \ar@{->}[u]\ar@{->}[r] &k\ar@{->}[u]
}\]
with each  $\heartsuit$ a distinct new object satisfying $\mu(\heartsuit) = 1$.

Assume that $\mcd_{n-1, k}$ has been constructed.
For each morphism $h$ of $\mcd_{n-1, k}$ with $\mu(h) = n -1$ we attach a new arrow $p_{n-1, n} \bullet h$ satisfying
\[ s(p_{n-1, n}\bullet h)=s(h) \text{  and  } \mu(p_{n-1, n} \bullet h)= n \]
whose target is a new object we call $(n,h)$. In this way we form a category $\mcd_{n,k}$, with $\mu(g)\leq n$ for every morphism $g$ in this category.

Two examples of this construction for low $n$ and $k$ are given in Figures~\ref{F:twotwo} and \ref{F:threeone}.
Note that when we build $\mcd_{n,k}$ we add an arrow exactly for those $h$ which are not the source of a nontrivial morphism.

\begin{figure}
\includegraphics[width=0.25\textwidth]{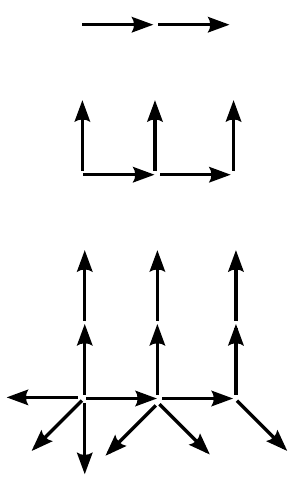}
\caption{The categories $\mcd_{0,2}$, $\mcd_{1,2}$, and $\mcd_{2,2}$ }\label{F:twotwo}
\end{figure}

\begin{figure}
\includegraphics[width=0.55\textwidth]{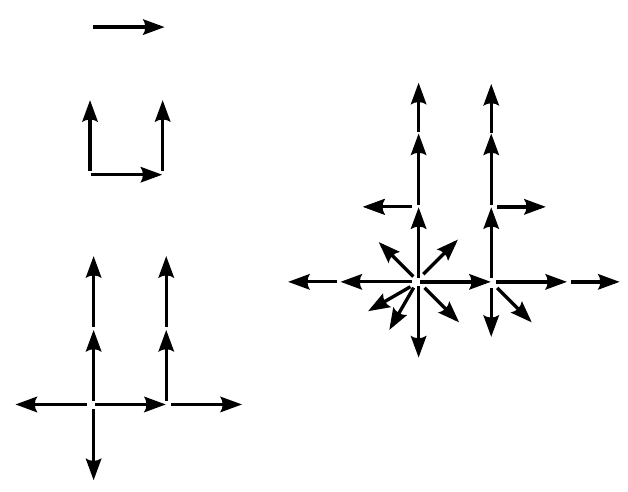}
\caption{The categories $\mcd_{0,1}$, $\mcd_{1,1}$, $\mcd_{2,1}$, and $\mcd_{3,1}$ }\label{F:threeone}
\end{figure}

The rooted action of $\mcc_{n,k}$ on $\mcd_{n,k}$ is given by
\[ p_{\ell, \ell'} \bullet h = p_{\ell' - 1, \ell'} \bullet \left( \dots \bullet \left( p_{\ell, \ell+1} \bullet h \right) \right). \]

\begin{defn}
We write $\bi{n}{k}$ for the above rooted action of $\mcc_{n,k}$ on $\mcd_{n,k}$. We define $\Delta \circlearrowright \Delta$ to be the full subcategory of $\catact$ with object set
$\setm{ \bi{n}{k} }{n\geq 0, k\geq -1}$.
\end{defn}

The following proposition is key to understanding maps in $\Delta \circlearrowright \Delta$.

\begin{prop}\label{P:universalproperty}
Suppose that there is a rooted action of $\mca$ on $\mcb$. Then maps
\[ X: \bi{n}{k} \to \mca \circlerightarrow \mcb \]
in $\catact$ are in bijection with pairs of functors
\begin{align*}
\alpha: [n] &\to \mca \\
\beta: [k] &\to \mcb
\end{align*}
satisfying $\mu(\beta(k)) = \alpha(0)$.
\end{prop}

\begin{proof}
A morphism $X$ in $\catact$ determines functors
\begin{gather*}
\alpha: [n] \to \mcc_{n,k} \overset{X^c}\to \mca \\
\beta: [k] \to \mcd_{n,k} \overset{X^d}\to \mcb
\end{gather*}
such that $\mu\beta(k) = \mu X^d(k) = X^c (0) = \alpha(0)$.

On the other hand, suppose we have a pair $\alpha$ and $\beta$ with $\mu(\beta(k)) = \alpha(0)$. Extend $\alpha$ to a functor $X^c: \mcc_{n,k}\to \mca$, defined on objects by
\begin{align*}
X^c(i) &= \alpha(i) & &0\leq i \leq n \\
X^c (*_w) &= \mu(\beta(w)) & &0\leq w < k.
\end{align*}
Write $X_0^d$ for $\beta: [k] = \mcd_{0,k} \to \mcb$.
We have a filtration $[k] = \mcd_{0,k} \hookrightarrow \mcd_{1,k} \hookrightarrow \cdots \hookrightarrow \mcd_{n,k}$, and we inductively define functors $X^d_\ell: \mcd_{\ell, k} \to \mcb$. The functor $X_\ell^d$ needs only to be defined on the new arrows $p_{\ell-1, \ell} \bullet h$, and must satisfy
\[ X_{\ell}^d (p_{\ell-1, \ell} \bullet h) = \alpha(p_{\ell-1, \ell}) \bullet X_{\ell-1}^d (h). \]
We define $X^d$ to be $X_n^d: \mcd_{n,k} \to \mcb$. By construction, the pair $(X^c, X^d)$ is a map of actions.
\end{proof}

Henceforth we always use the notation $(\alpha, \beta)$ for maps in $\Delta \circlearrowright \Delta$. For a map $X=(\alpha,\beta) : \bi{n}{k} \to \bi{m}{\ell}$ we write
\begin{equation} \begin{aligned}
\hat \alpha &:= X^c:  \mcc_{n,k} \to \mcc_{m,\ell} \\
\hat \beta &:= X^d: \mcd_{n,k} \to \mcd_{m,\ell}.
\end{aligned} \end{equation}

We note that a map $\bi{n}{k} \to \bi{m}{\ell}$ is not simply a pair of maps $[n] \to [m]$ and $[k] \to [\ell]$. The target of these maps should be $\mcc_{m,\ell}$ and $\mcd_{m,\ell}$, respectively. As an example, see Figure~\ref{F:actionex}, where the map $\alpha$ is given by $d^0: [1] \to [2]$.

\begin{figure}
\includegraphics[width=0.55\textwidth]{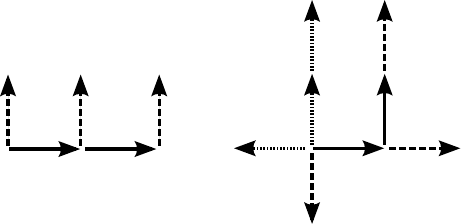}
\caption{A map $\bi{2}{2} \to \bi{3}{1}$}\label{F:actionex}
\end{figure}

We are now at a point where we can begin to talk about the \emph{nerve} of a rooted action.
Recall that the nerve of a category $\mathcal{C}$ is the simplicial set defined by $\nerve(\mathcal{C})_n = \Hom_{\cat}([n], \mathcal{C})$, together with the usual structure maps.
The resulting functor $\nerve\colon \cat \to \Set^{\Delta^{op}}$ is well-known to be fully faithful.
Similarly, there is a functor
\begin{align*}
 \nerve\colon \catact &\to \Set^{\Delta \circlearrowright \Delta^{op}}\\
\nerve(\mca \circlerightarrow \mcb)_{\bi{n}{k}} &= \Hom_{\catact} (\bi{n}{k},  \mca \circlerightarrow \mcb).
\end{align*}

\begin{prop}\label{P:ffnerve}
The functor $\nerve\colon \catact \to \Set^{\Delta \circlearrowright \Delta^{op}}$ is fully faithful.
\end{prop}
\begin{proof}
We use two inclusions $i_1, i_2 \colon \Delta \hookrightarrow \Delta \circlearrowright \Delta$ given on objects by
\begin{equation*}
i_1([n]) = \bi{n}{-1} \qquad i_2([k]) = \bi{0}{k}
\end{equation*}
and on morphisms by the characterization in Proposition~\ref{P:universalproperty}.
Let $\catact \to \cat \times \cat$ be the functor which is given on objects by $\mca \circlearrowright \mcb \mapsto (\mca, \mcb)$; then the diagram
\[
    \xymatrix@C=3cm{
        \catact \ar@{->}[r]^-{\nerve} \ar@{->}[d] &
        \Set^{\Delta \circlearrowright \Delta^{op}} \ar@{->}[d]^{i_1^* \times i_2^*} \\
        \cat \times \cat \ar@{->}[r]^-{\nerve \times \nerve} &
        \Set^{\Delta^{op}} \times \Set^{\Delta^{op}}
    }
\]
commutes. The bottom arrow in this diagram is fully faithful, and since maps in $\catact$ are determined by the underlying functors, the arrow on the left is faithful. Thus $\nerve \colon \catact \to \Set^{\Delta \circlearrowright \Delta^{op}}$ is faithful as well.

To see that the nerve functor is full, notice that
\[
    \nerve(\mca \circlearrowright \mcb)_{\bi{n}{k}} \cong \nerve(\mca)_n 
    \underset{A_0}{\times}
    \nerve (\mcb)_k
\]
by Proposition~\ref{P:universalproperty}, where the pullback is taken over the maps
\begin{align*}
    \mu \underbrace{d_0\dots d_0 }_k:  \nerve (\mcb)_k &\to  \nerve(\mcb)_0 = B_0 \to A_0 \\
    d_1 \dots d_{n-1}d_n: \nerve(\mca)_n &\to \nerve(\mca)_0 = A_0.
\end{align*}
Thus a map
$\nerve(\mca \circlearrowright \mcb) \to \nerve(\mca' \circlearrowright \mcb')$ is determined by by its action at $\bi{n}{-1}$ and $\bi{0}{k}$ where $n$ and $k$ ranges over the nonnegative integers; fullness then follows from the above diagram and fullness of the categorical nerve.
\end{proof}

\begin{defn}
The \emph{degree} of $\bi{n}{k}$ is
\[ d\bi{n}{k} = n + |\ob(\mcd_{n,k})|. \] Let $\Delta \circlearrowright \Delta^+$ be the wide subcategory of $\Delta \circlearrowright \Delta$ consisting of maps $(\alpha, \beta): \bi{n}{k} \to \bi{m}{\ell}$ such that the maps
\begin{align*}
\alpha: [n] &\to \mcc_{m,\ell}
& \beta: [k] &\to \mcd_{m,\ell}
\end{align*}
are injective on objects. Finally, let $\Delta \circlearrowright \Delta^-$ be the wide subcategory consisting of maps $(\alpha, \beta): \bi{n}{k} \to \bi{m}{\ell}$ such that $\alpha$ is surjective on the objects of $[m] \ci \mcc_{m,\ell}$, and
$ \hat \beta: \mcd_{n,k} \to \mcd_{m,\ell} $
is surjective on objects.
\end{defn}

\begin{lem}\label{L:structure} Let $(\alpha, \beta): \bi{n}{k} \to \bi{m}{\ell}$ be a map in $\Delta \circlearrowright \Delta$.
\begin{enumerate}
\item If all objects of $[m]$ are in the image of $\alpha$ and all objects of $[\ell]$ are in the image of $\beta$, then $(\alpha, \beta)$ is in $\Delta \circlearrowright \Delta^-$. \label{impliesminus}
\item If $(\alpha, \beta)$ is in $\Delta \circlearrowright \Delta^-$, then all objects of $[\ell]$ are in the image of $\beta$. \label{impliedminus}
\item If $(\alpha, \beta)$ is in $\Delta \circlearrowright \Delta^+$, then $\hat \alpha: \mcc_{n,k} \to \mcc_{m,\ell}$ and $ \hat \beta: \mcd_{n,k} \to \mcd_{m,\ell}$ are injective on objects. \label{impliedplus}
\end{enumerate}
\end{lem}

\begin{proof}
For \eqref{impliesminus}, we need to show that $\hat \beta: \mcd_{n,k} \to \mcd_{m,\ell}$ is surjective on objects. We proceed inductively: by assumption all objects of $\mcd_{0,\ell}$ are in the image of $\beta$.
Suppose that all objects of $\mcd_{x-1,\ell}$ are in the image of $\hat \beta$, and
consider the target $(x,h)$ of $p_{x-1, x} \bullet h$. We know by induction that $h=\hat \beta (h')$ for some $h'$.
Since all objects of $[m]$ are in the image of $\alpha$, there is an $i$  such that $\alpha(i-1) = x-1$ and $\alpha(i) = x$.
Then \[ p_{x-1, x} \bullet h = \alpha(p_{i-1,i}) \bullet \hat \beta(h') = \hat \beta( p_{i-1,i} \bullet h' ),\] so $(x,h)$ is in the image of $\hat \beta$.

Turning to \eqref{impliedminus}, we first show that $\ell = \beta(x)$ for some $0\leq x \leq k$.
 With the goal of finding a contradiction, suppose that $\ell \neq \beta(x)$ for all $0\leq x \leq k$.
Since $(\alpha, \beta) \in \Delta \circlearrowright \Delta^-$, we already know that $\ell$ is in the image of $\hat \beta: \mcd_{n,k} \to \mcd_{m,\ell}$ and we let \[ i=\min \setm{j}{\text{there exists a morphism $h$ with }\ell = \hat \beta(t(p_{j-1, j} \bullet h))}. \]
Pick a morphism $h$ with $\ell = \hat \beta(t(p_{i-1,i}\bullet h))$.
By construction of $\bi{m}{\ell}$ we know that $\mu(\ell) = 0$, so
\[ 0 = \mu (\hat \beta(t(p_{i-1,i}\bullet h))) = \mu(\hat\beta ( p_{i-1,i} \bullet h)) = \mu( \alpha(p_{i-1,i}) \bullet \hat \beta(h)) \]
so we see that  $\alpha(p_{i-1,i}) = \id_0$. But then
$\hat \beta ( p_{i-1,i} \bullet h) = \alpha(p_{i-1,i}) \bullet \hat \beta(h) = \hat \beta(h)$, so $\hat \beta(t(h)) = \hat \beta t (p_{i-1,i} \bullet h) = \ell$. Thus $\mu(h) = i-1 < i$ so we must have $i=1$ since $i$ was chosen minimally. But then $\mu(h) =0$, so $t(h) = k$ and $\hat \beta(t(h)) = \ell$, contrary to our assumption that $ \ell \neq \hat \beta(x)$ for $0\leq x \leq k$.

Likewise, if $1\leq j < \ell$, then we know $j$ is in the image of $\hat \beta$. But $\mu(\hat \beta (p_{i-1, i} \bullet h)) = \alpha(i) \neq *_j$ by assumption on $\alpha$. It follows that $j$ is in the image of $\hat \beta$ restricted to $\mcd_{0,k} = [k]$.

Finally, for \eqref{impliedplus}, we make use of two fundamental facts about $\bi{n}{k}$, both of which follow from construction of $\mcd_{n,k}$. The first is that if we order the objects of $\mcc_{n,k}$ as $*_0 < \dots < *_{k-1} < 0 < \dots <n$, then for any nontrivial morphism $h$ of $\mcd_{n,k}$ we have $\mu(s(h)) < \mu(t(h))$. The second is that if $(i, h) = (i,  h')$, then $h=h'$.

Making use of this first fact, we find since $\beta$ is injective on objects
\[ \hat \alpha(*_0) = \mu \beta (0) < \mu \beta(1)  \dots < \hat \alpha (*_{k-1}) = \mu(\beta(k-1)) < \mu(\beta(k)) = \alpha(0), \]
and we already knew that $\alpha(0) < \dots < \alpha(n)$, so $\hat \alpha$ is an increasing function on objects and thus injective.

By assumption we know that $\hat \beta$ is injective on the objects of $\mcd_{0,k}$. Assuming this map is injective on the objects of $\mcd_{i-1,k}$, we will show that it is injective on the objects of $\mcd_{i,k}$. All of the new objects in this category are of the form $(i, h)$, and since $\hat \alpha$ is increasing on objects, $\hat \beta (p_{i-1,i}  \bullet h)$ has strictly greater moment than any object in the image of $\hat \beta|_{\mcd_{i-1,k}}$. Thus we only need to show that if $\hat \beta (i, h) = \hat \beta (i, h')$ then $h=h'$.
But we have
\[ \hat \beta (i,h) = \hat \beta t(p_{i-1,i} \bullet h) = t \hat \beta (p_{i-1,i} \bullet h) = t(\alpha(p_{i-1,i}) \bullet \hat \beta(h)) \]
and therefore
\[ t(\alpha(p_{i-1,i}) \bullet \hat \beta(h)) = t(\alpha(p_{i-1,i}) \bullet \hat \beta(h')), \]
so (iterated use of) the second fundamental fact tells us that $\hat \beta(h) = \hat \beta(h')$. But $h$ and $h'$ are morphisms in $\mcd_{i-1,k}$, so $h=h'$. Thus $\hat \beta$ is also injective on $\mcd_{i,k}$.
\end{proof}

\begin{prop}
Given a map \[ (\alpha, \beta): \bi{n}{k} \to \bi{m}{\ell}, \]
there is a unique decomposition into a map of $\Delta \circlearrowright \Delta^-$ followed by a map of $\Delta \circlearrowright \Delta^+$.
\end{prop}

\begin{proof}
The first case we consider is when $\beta(k)$ is not one of the objects $0, \dots, \ell-1$ in $ \mcd_{m,\ell}$. This implies that $\mu(\beta(k)) \in [m]$, so $\alpha: [n] \to \mcc_{m,\ell}$ actually lands in $[m]$. Thus we have the factorization \[ \alpha: [n] \overset{}\twoheadrightarrow \underbrace{ [y] \hookrightarrow [m] \hookrightarrow \mcc_{m,\ell} }_{\alpha^+}\]
for some $y$ as given by the Reedy structure on $\Delta$.
We let $z+1$ be the number of objects in the image of $\beta: [k] \to \mcd_{m,\ell}$. The full subcategory generated by these objects must be a linear tree, since $[k]$ is linear and $\mcd_{m,\ell}$ is generated by a tree. We then have a factorization
\[ \beta: [k] \overset{}{\twoheadrightarrow} [z] \overset{\beta^+}\hookrightarrow \mcd_{m,\ell}.\]
We define
\begin{align*}
\alpha^-&: [n] \twoheadrightarrow  [y] \hookrightarrow \mcc_{y,z} \\
\beta^-&: [k] \twoheadrightarrow  [z] \hookrightarrow \mcd_{y,z}.
\end{align*}

We claim that $(\alpha, \beta)$ decomposes as
\[ (\alpha^+, \beta^+) \circ (\alpha^-, \beta^-): \bi{n}{k} \to \bi{y}{z} \to \bi{m}{\ell}. \] We have $\mu(\beta^-(k)) = \mu(z) = 0 = \alpha^-(0)$ since  $\beta^-$ and $\alpha^-$ surject onto $[z]$ and $[y]$. Furthermore,  $\mu(\beta^+(z)) = \mu(\beta(k)) = \alpha(0) = \alpha^+(0)$. Thus, by Proposition~\ref{P:universalproperty}, $(\alpha^+, \beta^+)$ and $(\alpha^-, \beta^-)$ are morphisms in $\Delta \circlearrowright \Delta$.

Notice that any decomposition must be this one. The definition of $\alpha^+$ and $\alpha^-$ is forced by the definition of $\Delta \circlearrowright \Delta^+$ and $\Delta \circlearrowright \Delta^-$. The definition of $\alpha^+$ and $\alpha^-$  then forces the definition of $\beta^+$ and $\beta^-$ by Lemma~\ref{L:structure}(\ref{impliesminus},\ref{impliedminus}).

The map $(\alpha^+, \beta^+)$ is in $\Delta \circlearrowright \Delta^+$ by definition of this category. Lemma~\ref{L:structure}\eqref{impliesminus} implies that $(\alpha^-, \beta^-)$ is in $\Delta \circlearrowright \Delta^-$.

We still must consider the case when $\beta(k)$ is one of the objects $0, \dots, \ell-1$. Then
\[ \alpha(0) = *_{\beta(k)}, \]
so $\alpha$ factors as $[n] \to [0] \overset{\alpha^+}{\to} \mcc_{m,\ell}$. We also have the factorization
\[ \beta: [k] \twoheadrightarrow \underbrace{[z] \hookrightarrow [\ell] \hookrightarrow \mcd_{m,\ell}}_{\beta^+}.\]
As before, we define $\alpha^-: [n] \to [0] \to \mcc_{0,z}$ and $\beta^-: [k] \to [z] \to \mcd_{0,z}$.
Thus we have the factorization \[ \bi{n}{k} \twoheadrightarrow \bi{0}{z} \hookrightarrow \bi{m}{\ell} \]
since \begin{gather*}
\mu (\beta^+(z)) = \mu(\beta(k)) = *_{\beta(k)} = \alpha(0) = \alpha^+(0) \\
\mu(\beta^-(k)) = \mu(z) = 0 = \alpha^-(0).
\end{gather*}
The fact that $(\alpha^\pm,\beta^\pm)$ are in $\Delta \circlearrowright \Delta^\pm$ and that this decomposition is unique follows as in the previous case.
\end{proof}

Recall that the degree of $\bi{n}{k}$ is defined to be
\[ d\bi{n}{k} = n + |\ob(\mcd_{n,k})|. \]

\begin{prop}\label{P:catoncatplus}
If a map $(\alpha, \beta): \bi{n}{k} \to \bi{m}{\ell}$ is in $\Delta \circlearrowright \Delta^+$, then
\[ d\bi{n}{k} \leq d \bi{m}{\ell} \] with equality holding if and only if $(\alpha,\beta)$ is an identity map.
\end{prop}
\begin{proof}
By Lemma~\ref{L:structure}\eqref{impliedplus} we know that
\[ |\ob(\mcd_{n,k})| \leq |\ob(\mcd_{m,\ell})|. \]
Injectivity of $\alpha$ implies that $\alpha$ lands in $m$ whenever $n>0$, so
\[ n \leq m. \]
This establishes the desired inequality.

We now check that if $d\bi{n}{k} = d\bi{m}{\ell}$ that $\bi{n}{k} = \bi{m}{\ell}$.
Equality here implies that $|\ob(\mcd_{n,k})| = |\ob(\mcd_{m,\ell})|$ and $n=m$.
We know $\hat \alpha: \mcc_{n,k} \to \mcc_{m,\ell}$ is injective on objects by Lemma~\ref{L:structure}\eqref{impliedplus}, so $k\leq \ell$ since these categories have $k+n+1$ and $\ell+m+1 = \ell+n+1$ objects, respectively.
If $k < \ell$, then there is an element $j$ of $[\ell]$ with $\hat \beta (i,h) = j$ and $i>0$, so $\alpha(i) = \mu(j) = *_j$.
But this is impossible by injectivity of $\hat \alpha$. (The only exception is when $n=m=0$, in which case there are no objects $(i,h)$ with $i>0$.)
\end{proof}

\begin{prop}\label{P:catoncatminus}
If a map $(\alpha, \beta): \bi{n}{k} \to \bi{m}{\ell}$ is in $\Delta \circlearrowright \Delta^-$, then
\[ d\bi{n}{k} \geq d \bi{m}{\ell} \] with equality holding if and only if $(\alpha,\beta)$ is an isomorphism.
\end{prop}
\begin{proof}
We know that $|\ob(\mcd_{n,k})| \geq | \ob(\mcd_{m,\ell})|$ by definition of $\Delta \circlearrowright \Delta^-$. Since $\alpha$ (not $\hat \alpha$) surjects onto the objects of $[m]$ we then have $n\geq m$. This establishes the inequality.

Suppose that the degrees  $d\bi{n}{k}$ and $d\bi{m}{\ell}$ are equal, whence $n=m$ and $\hat \beta: \mcd_{n,k} \to \mcd_{m,\ell}$ is a bijection on objects.
By Lemma~\ref{L:structure}\eqref{impliedminus} we have $k\geq \ell$.
We have $\alpha(0) =0$, so $\beta(k) = \ell$ since $k$ and $\ell$ are the only objects of moment $0$. But then $\beta(i \to k) = \beta(i) \to \beta(k) = \ell$ so $\beta(i)$ must be in $[\ell]$ as well, for all $0\leq i \leq k$. We have $\hat \beta$ injective so $k \leq \ell$. We have established that $k=\ell$, and that $n=m$.
Thus $\bi{n}{k} = \bi{m}{\ell}$.
\end{proof}

The result of the previous three propositions is the following.

\begin{thm}
The category $\Delta \circlearrowright \Delta$ is a (strict) Reedy category. 
\end{thm}

\newcommand{\R}{\mathcal{R}}
We use the following characterization of elegance from \cite[3.4]{bergnerrezk}. Let $\mathcal{R}$ be a Reedy category and
$F: \mathcal{R} \to \Set^{\mathcal{R}^{op}}$ be the Yoneda functor with $F(r) = \Hom(-,r)$. We say $\mathcal{R}$ is \emph{elegant} if every pair of maps $\sigma_i: r\to a_i$, $i=1,2$ in $\mathcal{R}^-$ extends to a commutative square in $\mathcal{R}^-$ which is a strong pushout in $\R$. In other words, there exist $\tau_i: a_i \to b$ in $\mathcal{R}^-$ such that $\tau_1 \sigma_1 = \tau_2 \sigma_2$ and such that
\begin{equation}
\begin{gathered} \xymatrix{
Fr \ar@{->}[r]^{\sigma_1} \ar@{->}[d]_{\sigma_2} & Fa_1 \ar@{->}[d]^{\tau_1} \\
Fa_2 \ar@{->}[r]_{\tau_2} & Fb
} \end{gathered} \label{D:pushout}
\end{equation}
is a pushout square in $\Set^{\R^{op}}$.

\begin{lem} \label{L:deltaplus}
The category $\Delta_\diamondsuit$, a skeleton of the category of finite ordered sets with objects $\set{[-1]:=\varnothing, [0], [1], \dots}$, is an elegant Reedy category.
\end{lem}

\begin{proof}
The degree function on $\Delta_\diamondsuit$ is given by $d[n] = n+1$. The object $[-1]$ is initial. The morphisms of the category $\Delta_{\diamondsuit}^+$ are those of $\Delta^+$ together with all of the maps $[-1] \to [n]$. The morphisms of $\Delta_{\diamondsuit}^-$ are those of $\Delta^-$ along with $[-1] \to [-1]$. Since $[-1] \to [-1]$ is the only map with target $[-1]$, the factorizations follow as in $\Delta$, along with $[-1] \to [n]$ having the factorization $[-1] \twoheadrightarrow [-1] \hookrightarrow [n]$.  We only need to check the degree conditions for the new maps $[-1] \to [n]$, and these follow immediately since $d[-1] = 0 \leq n+1 = d[n]$, with equality holding only when $n=-1$.

Elegance essentially follows from elegance of $\Delta$ \cite{bergnerrezk}. The only map in $\Delta_\diamondsuit^-$ which involves $[-1]$ is the identity on $[-1]$, and all other maps are in $\Delta$. So we merely need to check the property when $\sigma_1 = \sigma_2 = \id_{[-1]}$, and we set $b=[-1]$ so the diagram~\ref{D:pushout} is indeed a pushout diagram.
\end{proof}

We now strengthen Lemma~\ref{L:structure}\eqref{impliedminus}. 
\begin{lem}\label{L:precisetarget}
We have the following:
\begin{enumerate}
\item Using the order $*_0 < *_1 < \dots < *_{k-1} < 0 < 1 < \dots < n$ of the objects of $\mcc_{n,k}$, if $a\to b$ is a morphism of $\mcd_{n,k}$ then $\mu(a) \leq \mu(b)$.\label{orderofobjects}
\item If $(\alpha,\beta): \bi{n}{k} \to \bi{m}{\ell}$ is in $\Delta \circlearrowright \Delta^-$, then every object in the image of $\beta: [k] \to \mcd_{m,\ell}$ is in $[\ell]$.\label{precisetarget}
\end{enumerate}
\end{lem}

\begin{proof}
To prove \eqref{orderofobjects}, first observe that if $a \to b$ is an identity or $a\to b$ is in $[k]$, then the result is immediate. If $a\to b$ is a morphism with $b$ an object of $\mcd_{0,k} = [k]$, then $a$ is an object of $\mcd_{0,k}$ since there are no generating arrows $(i,h) \to j$. We proceed by induction on the moment of the map $a \to b$. If $b=(i,h)$, there is only one generating morphism with target $b$, $s(h) \to (i,h)$. Thus if $a\to b$ is not an identity, then it factors as $a\to s(h) \to b$ and $s(h)$ is an object of $\mcd_{i-1,k}$ by construction of $\mcd$. We already have the result for $i-1$, so $\mu(a) \leq \mu(s(h)) = i-1 < i = \mu(b)$.

It remains to prove \eqref{precisetarget}.  Since have arrows
\[ \beta(0) \to \beta(1) \to \dots \to \beta(k-1) \to \beta(k), \]
by \eqref{orderofobjects} we have
\[ \mu\beta(0) \leq \mu\beta(1) \leq \dots \leq \mu\beta(k-1) \leq \mu\beta(k) = \alpha(0) = 0.\]
The last equality holds since $\alpha$ surjects onto $[m]$. Since
\[ \mu\beta(0), \dots, \mu\beta(k) \in \set{*_0, \dots, *_{\ell-1}, 0} \]
we have that $\beta(0), \dots, \beta(k) \in \set{0,1,\dots, \ell}$.
\end{proof}

\begin{thm}
The Reedy category $\Delta \circlearrowright \Delta$ is elegant.
\end{thm}

\begin{proof}
Suppose that we have maps $(\alpha_i,\beta_i): \bi{n}{k} \to \bi{m_i}{\ell_i}$ in $\Delta \circlearrowright \Delta^-$ for $i=1,2$. By the definition of $\Delta \circlearrowright \Delta^-$ and the fact that $[n]$ is connected, we may consider $\alpha_i: [n] \to \mcc_{m_i,\ell_i}$ as a map $\alpha_i: [n] \to [m_i]$, which is in $\Delta^-$. We thus have a strong pushout square
\[ \xymatrix{
[n] \ar@{->}[r]^{\alpha_1} \ar@{->}[d]_{\alpha_2} & [m_1] \ar@{->}[d]^{\delta_1} \\
[m_2] \ar@{->}[r]_{\delta_2} & [w]
} \]
since $\Delta$ is elegant. We also consider $\beta_i: [k] \to \mcd_{m_i,\ell_i}$ as a surjective map $[k] \to [\ell_i]$ by Lemma~\ref{L:structure}\eqref{impliedminus} and Lemma~\ref{L:precisetarget}\eqref{precisetarget}. Since $\Delta_\diamondsuit$ is an elegant Reedy category by Lemma~\ref{L:deltaplus}, we have a strong pushout square
\[ \xymatrix{
[k] \ar@{->}[r]^{\beta_1} \ar@{->}[d]_{\beta_2} & [\ell_1] \ar@{->}[d]^{\gamma_1} \\
[\ell_2] \ar@{->}[r]_{\gamma_2} & [x]
} \]
in $\Delta_\diamondsuit$.

We claim that $(\delta_i, \gamma_i)$ is a morphism in $\Delta \circlearrowright \Delta^-$, $i=1,2$, and that the corresponding square is a strong pushout. Since $\delta_i(0) = 0$ and $\gamma_i(\ell_i) = x$ by surjectivity, we have that $\mu(\gamma_i(\ell_i)) = \mu(x) = 0 = \delta_i(0)$. Thus $(\delta_i,\gamma_i)$ is a map in $\Delta \circlearrowright \Delta$ by
Proposition~\ref{P:universalproperty}. It is in $\Delta \circlearrowright \Delta^-$ by Lemma~\ref{L:structure}\eqref{impliesminus}.

It is now left to show that the square
\begin{equation} \begin{gathered} \xymatrix{
F\bi{n}{k} \ar@{->}[r]^{(\alpha_1,\beta_1)} \ar@{->}[d]_{(\alpha_2,\beta_2)} & F\bi{m_1}{\ell_1} \ar@{->}[d]^{(\delta_1,\gamma_1)} \\
F\bi{m_2}{\ell_2} \ar@{->}[r]_{(\delta_2,\gamma_2)} & F\bi{w}{x}
} \end{gathered} \label{D:presheafsquare} \end{equation}
is a pushout square in $\Set^{\Delta \circlearrowright \Delta^{op}}$. It is enough to show that \[ \xymatrix{
\Hom(\bi{y}{z}, \bi{n}{k}) \ar@{->}[r]^{(\alpha_1,\beta_1)} \ar@{->}[d]_{(\alpha_2,\beta_2)} & \Hom(\bi{y}{z}, \bi{m_1}{\ell_1}) \ar@{->}[d]^{(\delta_1,\gamma_1)} \\
\Hom(\bi{y}{z}, \bi{m_2}{\ell_2} )\ar@{->}[r]_{(\delta_2,\gamma_2)} & \Hom(\bi{y}{z}, \bi{w}{x})
} \]
is a pushout diagram in $\Set$ for each object $\bi{y}{z}$ in $\Delta \circlearrowright \Delta$.

We have
\begin{equation}
\begin{aligned}
 \Hom(\bi{y}{z}, \bi{a}{b}) & = \setm{\sigma \times \tau}{\sigma: [y]\to [a], \tau: [z] \to [b], \sigma(0) =\mu\tau(z) }  \\
& \subseteq \Hom([y], [a]) \times \Hom([z],[b])
\end{aligned} \label{E:splituphomthing}
\end{equation}
and we compute that the pushout should be
\begin{equation}
\left[ \Hom(\bi{y}{z}, \bi{m_1}{\ell_1}) \amalg \Hom(\bi{y}{z}, \bi{m_2}{\ell_2} ) \right] / \sim \label{E:actualpushout} \end{equation}
where
\[  (\sigma_1,\tau_1) \sim (\sigma_2, \tau_2) \text{ when } \sigma_1 \alpha_1 = \sigma_2 \alpha_2 \text{ and } \tau_1 \beta_1 = \tau_2 \beta_2.\]
This pushout is contained in
\begin{equation}
\left[\Hom([y],[m_1]) \times \Hom([z],[\ell_1]) \amalg \Hom([y],[m_2])\times \Hom([z],[\ell_2]) \right] / \sim \label{E:stretched}
\end{equation}
where $\sigma_1\times \tau_1 \sim \sigma_2\times \tau_2$ when $\sigma_1 \alpha_1 = \sigma_2 \alpha_2$  and  $\tau_1 \beta_1 = \tau_2 \beta_2,$
with the extra conditions that $\sigma_i(0) = \mu\tau(z)$.
We see that \eqref{E:stretched} is equal to
\[ \Hom([y],[w]) \times \Hom([z],[x]) \]
and by \eqref{E:splituphomthing} we have that
\[ \Hom(\bi{y}{z}, \bi{w}{x}) \subseteq \Hom([y],[w]) \times \Hom([z],[x]) \]
is equal to \eqref{E:actualpushout}. Thus when we evaluate the diagram of presheaves \eqref{D:presheafsquare} on any object of $\Delta \circlearrowright \Delta$ we get a pushout, so \eqref{D:presheafsquare} is itself a pushout. Hence $\Delta \circlearrowright \Delta$ is elegant.
\end{proof}

\section{Rooted actions on colored operads}\label{S:catoperad}

Consider a category $\mcc$ and a colored operad $\mco$.  Since a colored operad can be equivalently regarded as a multicategory, its morphisms have any finite number of inputs (including possibly no inputs) and one output.  Using this perspective, we make extend the definition from the previous section as follows.

\begin{defn}
A \emph{rooted action} of a category $\mcc$ on a colored operad $\mco$ is an action of $\mcc$ on the set $\mor (\mco)$ of morphisms of $\mco$ satisfying the additional axioms:
\begin{itemize}%{labelitemi}
\renewcommand{\labelitemi}{$\cdot$}
\item if $g, g_1, \dots, g_k$ are in $\mor(\mco)$, then $\mu(\gamma(g; g_1, \dots, g_k)) = \mu(g)$, where $\gamma$ is the operadic composition,
\item $s(f\bullet g) = s(g)$ as ordered lists of colors of $\mco$.
\end{itemize}
\end{defn}

If $\mco$ is additionally a symmetric colored operad, then for any element $\sigma$ of the appropriate symmetric group, we furthermore require that
\begin{equation} \sigma^* (f\bullet g) = f\bullet (\sigma^* g). \label{E:symaction}\end{equation}
In both the symmetric and nonsymmetric cases, we write $\mcc \overset\bullet\circlearrowright \mco$ for such a rooted action.
A map
\[ X: \mca\overset\bullet\circlearrowright\mcp \to \mca' \overset{\blacktriangledown}\circlearrowright\mcp' \]
consists of a functor $X^c: \mca \to \mca'$ and an operad map $X^d: \mcp \to \mcp'$ which satisfy $X^c \mu  = \mu X^d $ and $X^d (a \bullet f) = X^c(a) \blacktriangledown X^d(f)$, where $a$ is a morphism of $\mca$ and $f$ is an operation of $\mcp$.  With such morphisms, we have $\opact$, the category whose objects are rooted actions on symmetric operads, and $\opactns$, the category of rooted actions on nonsymmetric operads.

\begin{example}
We point out a couple of important examples of rooted actions on operads, where the acting category is a group.
Suppose that $X$ is a $G$-space. Then the usual endomorphism operad $\mathcal E_X$ admits a rooted action by $G$. It is defined, for $f\in \mathcal E_X (n) = \map(X^{\times n}, X)$ by
\[ (g\bullet f) (x_1, \dots, x_n) = g\bullet (f(x_1, \dots, x_n)). \]
In fact, if $X$ is a deformation retract of another space $Y$, then $\mathcal E_Y$ inherits a rooted action by $G$. Another example (in the topological setting) is the framed-little disks operad $f\mathcal{D}_2$, which admits a rooted action of the circle by rotation of the outer disk.
\end{example}

We now imitate the construction of the $\bi{n}{k}$ from the previous section.
For each $n\geq 0$ and each planar tree $S$ (i.e., object of $\Omega_p$) we define an object $\bi{n}{S}$ of $\opactns$.  Let $\mathbf{r}=\mathbf{r}_S$ be the root of $S$.  First, we define $\mcc_{n,S}$ as the free category with object set
\[ \setm{*_e}{e\in E(S), e\neq \mathbf{r}} \sqcup \set{0, 1, \dots, n}  \cong (E(S) \sqcup \ob [n]) / (\mathbf{r} \sim 0) \]
with morphisms generated by $p_{i,i+1}: i\to i+1$. As in Section~\ref{S:catoncat}, we write
\[ p_{i,i+j}: i \to i+j \]
for the unique map when $j\geq 1$.

The construction of the operad $\mco_{n,S}$ is made inductively and mirrors the construction of $\mcd_{n,k}$ in the previous section.  For the base case, we set
\[ \mco_{0,S} = \Omega_{p}(S). \]
To build $\mco_{n,S}$ from $\mco_{n-1,S}$, we add new colors and operations as follows.
For each $h\in \mor \mco_{n-1,S}$ with $\mu(h)=n-1$, we add a new color $(n,h)$ along with a new generating morphism $p_{n-1,n} \bullet h$ such that
\begin{align*}
\mu(p_{n-1,n} \bullet h) &= n & s(p_{n-1,n} \bullet h) &= s(h) \\
& & t(p_{n-1,n} \bullet h) &= (n,h).
\end{align*}
The action is given by
\[ p_{\ell, \ell'} \bullet h = p_{\ell' - 1, \ell'} \bullet \left( \dots \bullet \left( p_{\ell, \ell+1} \bullet h \right) \right). \]

We refer to such morphisms as the \emph{generating morphisms} of $\mco_{n,S}$, and denote the set of such by $gen(n,S)$. Observe that this set consists of the set of vertices of $S$ together with one morphism $p_{i-1,i} \bullet h$ with $\mu(h) = i-1$, for each $1 \leq i \leq n$.  In particular, the number of these morphisms is
\[ |gen(n,S)| = |V(S) \sqcup \col(\mco_{n,S}) \setminus E(S)|. \]

\begin{defn}
Let $\bi{n}{S}$ be the above rooted action of $\mcc_{n,S}$ on $\mco_{n,S}$. We define the category $\dcop$ to be the full subcategory of $\opactns$ with
object set
\[ \setm{\bi{n}{S}}{n\geq 0, S \text{ a planar rooted tree}} \sqcup \set{\bi{n}{\varnothing}}. \]
\end{defn}

Notice that the color set $C=\col(\mco_{n,S})$ has a natural partial order $\prec$. On $\mco_{0,S} = \Omega_p(S)$, this partial order is that of the edges of the tree $S$, with the root $\mathbf{r}$ the maximal element.  The set of minimal elements consists of the leaves, together with edges attached to vertices with no inputs. The order on $\col(\mco_{n,S})$ extends that on $\col(\mco_{n-1,S})$, with $c \prec (n,h)$ for every $c$ which is an input to $h$, and $(n,h)$ is incomparable to $(n,h')$ for $h\neq h'$.

\begin{rem}\label{R:symmetricversion}
One can construct formal rooted actions on the symmetric operad $\Omega(S)$ in much the same way.  Aside from beginning with $\Omega(S)$ rather than $\Omega_p(S)$ at level $0$, we must also modify the inductive step.  We add a new color for each orbit class $[h]$ with $\mu(h)=n-1$, together with a corresponding new morphism $p_{n-1,n} \bullet h$ with $t(p_{n-1,n} \bullet h) = (n,[h])$.  Furthermore, we must also specify that $\sigma^*(p_{n-1,n} \bullet h) = p_{n-1,n} \bullet (\sigma^*(h))$.
Alternatively, we could simply symmetrize the nonsymmetric operad $\mco_{n,S}$, with the same result. We revisit the symmetric case in Section~\ref{S:symmetric}.
\end{rem}

\begin{prop}\label{P:catopuniversalproperty}
Suppose that there is a rooted action of a category $\mca$ on a (nonsymmetric) operad $\mcp$. Then a map
\[ X: \bi{n}{S} \to \mca \betweener \mcp \]
in $\opactns$ is equivalent to a pair of morphisms
\begin{align*}
\alpha: [n] &\to \mca \\
\beta: \Omega_p(S) &\to \mcp
\end{align*}
satisfying $\mu(\beta(\mathbf{r})) = \alpha(0)$.
\end{prop}

\begin{proof}
The proof follows as the one for Proposition~\ref{P:universalproperty}.
\end{proof}

Henceforth, when we deal with morphisms in $\dcop$, we always write a map as $(\alpha, \beta): \bi{n}{S} \to \bi{m}{R}$, and correspondingly use the shorthand
\begin{equation}
\begin{gathered}
\hat \alpha:= X^c: \mcc_{n,S} \to \mcc_{m,R} \\
\hat \beta:= X^d: \mco_{n,S} \to \mco_{m,R}
\end{gathered}\label{E:alphabetahat}
\end{equation}
for the components of the morphism $X$.

\begin{defn}
The \emph{degree} of $\bi{n}{S}$ is
\begin{equation}
d\bi{n}{S} = n + |gen(n,S)| = n + |\col(\mco_{n,S})| + |V(S)| - |E(S)|. \label{E:operaddegree}
\end{equation}

Let $\dcop^+$ be the wide subcategory of $\dcop$ consisting of maps $(\alpha, \beta): \bi{n}{S} \to \bi{m}{R}$ such that the maps
\[ \alpha: [n] \to \mcc_{m,R} \text{  and  } \beta: \Omega_p(S) \to \mco_{m,R} \]
are injective on objects and colors, respectively. Finally, let $\dcop^-$ be the wide subcategory consisting of maps $(\alpha, \beta): \bi{n}{S} \to \bi{m}{R}$ such that $\alpha$ is surjective on the objects of $[m] \ci \mcc_{m,R}$,
$ \hat \beta: \mco_{n,S} \to \mco_{m,R} $
is surjective on colors, and $\beta$ takes leaves of $S$ to leaves of $R$.
\end{defn}

\begin{lem}\label{L:structureop} Let $(\alpha, \beta): \bi{n}{S} \to \bi{m}{R}$ be a map in $\dcop$.
\begin{enumerate}
\item If all objects of $[m]$ are in the image of $\alpha$, all colors of $R$ are in the image of $\beta$, and $\beta$ takes leaves of $S$ to leaves of $R$, then $(\alpha, \beta)$ is in $\dcop^-$. \label{impliesminusop}
\item If $(\alpha, \beta)$ is in $\dcop^-$, then all colors of $R$ are in the image of $\beta$. \label{impliedminusop}
\item If $(\alpha, \beta)$ is in $\dcop^+$, then $\hat \alpha: \mcc_{n,S} \to \mcc_{m,R}$ and $ \hat \beta: \mco_{n,S} \to \mco_{m,R}$ are injective on objects and colors, respectively. \label{impliedplusop}
\end{enumerate}
\end{lem}

\begin{proof}
For \eqref{impliesminusop}, we need to show that $\hat \beta: \mco_{n,S} \to \mco_{m,R}$ is surjective on colors. We proceed inductively: by assumption all colors of $\mco_{0,R}$ are in the image of $\beta$.
Suppose that all colors of $\mco_{x-1,R}$ are in the image of $\hat \beta$, and
consider the target $(x,h)$ of $p_{x-1, x} \bullet h$. We know by induction that $h=\hat \beta (h')$ for some $h'$.
Since all objects of $[m]$ are in the image of $\alpha$, there is an $i$  such that $\alpha(i-1) = x-1$ and $\alpha(i) = x$.
Then \[ p_{x-1, x} \bullet h = \alpha(p_{i-1,i}) \bullet \hat \beta(h') = \hat \beta( p_{i-1,i} \bullet h' ),\] so $(x,h)$ is in the image of $\hat \beta$.

Turning to \eqref{impliedminusop}, we first show that for the root $\mathbf{r}$ of $R$, that $\mathbf{r} = \beta(x)$ for some $x\in E(S)$.
Suppose that $\mathbf{r} \neq \beta(x)$ for all $x\in E(S)$, which we claim leads to a contradiction.
Since $(\alpha, \beta) \in \dcop^-$, we already know that $\mathbf{r}$ is in the image of $\hat \beta: \mco_{n,S} \to \mco_{m,R}$ and we let \[ i=\min \setm{j}{\text{there exists a morphism $h$ with }\mathbf{r} = \hat \beta(t(p_{j-1, j} \bullet h))}. \]
Pick a morphism $h$ with $\mathbf{r} = \hat \beta(t(p_{i-1,i}\bullet h))$.
By construction of $\bi{m}{R}$ we know that $\mu(\mathbf{r}) = 0$, so
\[ 0 = \mu (\hat \beta(t(p_{i-1,i}\bullet h))) = \mu(\hat\beta ( p_{i-1,i} \bullet h)) = \mu( \alpha(p_{i-1,i}) \bullet \hat \beta(h)) \]
so we see that  $\alpha(p_{i-1,i}) = \id_0$. But then
$\hat \beta ( p_{i-1,i} \bullet h) = \alpha(p_{i-1,i}) \bullet \hat \beta(h) = \hat \beta(h)$, so $\hat \beta(t(h)) = \hat \beta t (p_{i-1,i} \bullet h) = \mathbf{r}$. Thus $\mu(h) = i-1 < i$ so we must have $i=1$ since $i$ was chosen minimally. But then $\mu(h) =0$, so $t(h) = \mathbf{r}_S$ (the root of $S$) and $\hat \beta(t(h)) = \mathbf{r}$, contrary to our assumption that $ \mathbf{r} \neq \hat \beta(x)$ for $x\in E(S)$.

Likewise, if $e \in E(R) \setminus \set{\mathbf{r}}$, then we know $e$ is in the image of $\hat \beta$. But $\mu(\hat \beta (p_{i-1, i} \bullet h)) = \alpha(i) \neq *_e \in \mcc_{m,R}$ by assumption on $\alpha$. It follows that $e$ is in the image of $\hat \beta$ restricted to $\mco_{0,S} = \Omega_p(S)$.

Finally, for \eqref{impliedplusop}, we make use of two fundamental facts about $\bi{n}{S}$, both of which follow from construction of $\mco_{n,S}$.
Note that $\ob \mcc_{n,S} = E(S) \setminus \set{\mathbf{r}} \sqcup \set{0,1,\dots, n}$ has a natural partial order $\prec$ induced from that on $E(S)$ and these $n+1$ integers.
Namely, $*_e\prec *_{e'}$ whenever $e$ lies above $e'$ in the tree $S$, $i\prec i'$ whenever $i< i'$, and $*_e\prec i$ for all $i$ and all $e$.
The first fundamental fact is that for any nontrivial morphism $h$ of $\mco_{n,S}$ and $c \in s(h)$, we have $\mu(c) \prec \mu(t(h))$. The second is that if $(i, h) = (i,  h')$, then $h=h'$.

We make use of this first fact. The map $\beta$ \emph{strictly} preserves the partial order since it is injective on objects. Furthermore, $\mu$ preserves the partial order. So if $e \prec e'$ in $E(S)$ then $\hat \alpha (*_e) = \mu\beta(e) \prec \mu \beta(e') = \hat \alpha (*_{e'})$. We already knew that $\alpha(i) \prec \alpha(i')$ for $i < i'$. Finally, for $e\neq \mathbf{r}$, we have $\hat \alpha(*_e) = \mu \beta(e) \prec \mu \beta(\mathbf{r}) = \alpha(0)$ so we see that $\hat \alpha$ strictly preserves this partial order, hence is injective.

By assumption we know that $\hat \beta$ is injective on the colors of $\mco_{0,k}$. Assume this map is injective on the colors of $\mco_{i-1,k}$, we will show that it is injective on the colors of $\mco_{i,k}$. All of the new objects in this category are of the form $(i, h)$, and since $\hat \alpha$ strictly preserves the partial order on objects, $\hat \beta (p_{i-1,i}  \bullet h)$ has strictly greater moment than any object in the image of $\hat \beta|_{\mco_{i-1,k}}$. Thus we only need to show that if $\hat \beta (i, h) = \hat \beta (i, h')$ then $h=h'$.
But we have
\[ \hat \beta (i,h) = \hat \beta t(p_{i-1,i} \bullet h) = t \hat \beta (p_{i-1,i} \bullet h) = t(\alpha(p_{i-1,i}) \bullet \hat \beta(h)) \]
and therefore
\[ t(\alpha(p_{i-1,i}) \bullet \hat \beta(h)) = t(\alpha(p_{i-1,i}) \bullet \hat \beta(h')), \]
so (iterated use of) the second fundamental fact tells us that $\hat \beta(h) = \hat \beta(h')$. But $h$ and $h'$ are morphisms in $\mco_{i-1,k}$, so $h=h'$. Thus $\hat \beta$ is also injective on $\mco_{i,k}$.
\end{proof}

\begin{prop}\label{P:catonopdecomposition}
Given a map $X=(\alpha, \beta): \bi{n}{S} \to \bi{m}{R}$ there is a unique decomposition into a map of $\dcop^-$ followed by a map of $\dcop^+$.
\end{prop}

\begin{proof}
\noindent We begin by proving the special case when $\beta(\mathbf{r})$ is not in
\[ \left[ E(R) \setminus \set{\mathbf{r}} \right] \ci \col (\mco_{m,R}). \]
Then $\mu(\beta(\mathbf{r})) \in [m]$, so $\alpha: [n] \to \mcc_{m,R}$ factors through the inclusion $[m] \hookrightarrow \mcc_{m,R}$.
We then have
\[ \alpha: [n] \overset{}\twoheadrightarrow \underbrace{ [y] \hookrightarrow [m] \hookrightarrow \mcc_{m,R} }_{\alpha^+}\]
from the Reedy factorization of $[n] \to [m]$ in $\Delta$.
By Lemma~\ref{L:opfactor} the map of operads $\beta: \Omega_p(S) \to \mco_{m,R}$ factors into a map which is surjective on objects followed by a map that is injective on objects:
\[ \Omega_p(S) \twoheadrightarrow \underbrace{\Omega_p(T) \hookrightarrow \mco_{m,R}}_{\beta^+}. \]
for some tree $T$.

We define
\begin{align*}
\alpha^-&: [n] \twoheadrightarrow  [y] \hookrightarrow \mcc_{y,T} \\
\beta^-&: \Omega_p(S) \twoheadrightarrow  \Omega_p(T) \hookrightarrow \mco_{y,T}.
\end{align*}

We claim that $(\alpha, \beta)$ decomposes as
\[ (\alpha^+, \beta^+) \circ (\alpha^-, \beta^-): \bi{n}{S} \to \bi{y}{T} \to \bi{m}{R}. \] We have $\mu(\beta^-(\mathbf{r}_S)) = \mu(\mathbf{r}_T) = 0 = \alpha^-(0)$ since  $\beta^-$ and $\alpha^-$ arise from maps which are surjective on colors of $\Omega_p(T)$ and $[y]$. Furthermore,  $\mu(\beta^+(\mathbf{r}_T)) = \mu(\beta(\mathbf{r}_S)) = \alpha(0) = \alpha^+(0)$. Thus, by Proposition~\ref{P:catopuniversalproperty}, $(\alpha^+, \beta^+)$ and $(\alpha^-, \beta^-)$ are morphisms in $\dcop$.

Notice that this decomposition is unique. The definition of $\alpha^+$ and $\alpha^-$ is forced by the definition of $\dcop^+$ and $\dcop^-$, which in turn forces the definition of $\beta^+$ and $\beta^-$ by Lemma~\ref{L:structureop}(\ref{impliesminusop},\ref{impliedminusop}).

The map $(\alpha^+, \beta^+)$ is in $\dcop^+$ by definition of this category. Lemma~\ref{L:structureop}\eqref{impliesminusop} implies that $(\alpha^-, \beta^-)$ is in $\dcop^-$.

It remains to consider the case when $\beta(\mathbf{r}_S)$ is in $E(R) \setminus \set{\mathbf{r}_R}$. Then
\[ \alpha(0) = \beta(\mathbf{r}_S) \in \mcc_{n,R} , \]
so $\alpha$ factors as $[n] \to [0] \overset{\alpha^+}{\to} \mcc_{m,R}$. We also have the factorization
\[ \beta: \Omega_p(S) \twoheadrightarrow \underbrace{\Omega_p(T) \hookrightarrow \Omega_p(R) \hookrightarrow \mco_{m,R}}_{\beta^+}\] from the  Reedy structure on $\Omega_p$.
As before, we define $\alpha^-: [n] \to [0] \to \mcc_{0,R}$ and $\beta^-: \Omega_p(S) \to \Omega_p(T) \to \mco_{0,T}$.
Thus we have the factorization \[ \bi{n}{S} \twoheadrightarrow \bi{0}{T} \hookrightarrow \bi{m}{R} \]
since \begin{gather*}
\mu (\beta^+(\mathbf{r}_T)) = \mu(\beta(\mathbf{r}_S))
= \alpha(0) = \alpha^+(0) \\
\mu(\beta^-(\mathbf{r}_S)) = \mu(\mathbf{r}_T) = 0 = \alpha^-(0).
\end{gather*}
To show that these maps are in $\dcop^\pm$, and that this decomposition is unique, we can use an argument as in the previous case.
\end{proof}

\begin{lem}\label{L:opfactor}
Suppose that $\beta: \Omega_p(S) \to \mco_{m,R}$ is a map of operads. Then there is a tree $T$ and a decomposition
\[ \Omega_p(S) \twoheadrightarrow \Omega_p(T) \hookrightarrow \mco_{m,R}\]
which is the unique factorization of the operad homomorphism $\beta$ into a map which is surjective on colors followed by a map which is injective on colors.
\end{lem}

Specifically, notice that the first map is a composition of degeneracies, and hence leaves of $S$ are mapped to leaves of $T$.

\begin{proof}
We say that a color $c'$ in $\mco_{m,R}$ \emph{lies over} a color $c$ if $c'$ is one of the inputs of a nontrivial morphism whose output is $c$. Fix a color $c_0$ in $\mco_{m,R}$.  We want to define a tree $T_0$ whose root is $c_0$ and whose edges are the colors lying over $c_0$. Let the set of edges $E(T_0) \ci \col (\mco_{m,R})$ be the set of all colors lying over $c_0$, together with $c_0$ itself.  The set of vertices $V(T_0)$ is given as follows. If $e\in E(T_0)$ and $e\in E(R)$ then we include in $V(T_0)$ the vertex $v\in V(R)$ which has $e$ as its output, provided this exists; moreover, if $e\in E(T_0)$ and $e$ is the output for a vertex $v$ in $R$ which has no inputs, we also include $v\in V(T_0)$. If $(i,h) \in E(T_0)$, then we include $(p_{i-1,i}\bullet h) \in V(T_0)$. These two sets determine a graph $T_0$. We define the input and output edges of a vertex $v\in V(R) \cap V(T_0)$ to be the input and output edges from the original graph $R$, whereas the input and output of $(p_{i-1,i} \bullet h)$ are $s(h)$ and $(i,h)$, respectively.

We claim that $T_0$ is a tree with root $c_0$. There is a partial order on $\col(\mco_{m,R})$ given by $c \prec c'$ precisely when $c$ lies above $c'$, and the induced partial order on $E(T_0) \ci \col(\mco_{m,R})$ has a unique maximal element $c_0$.  Hence, $T_0$ is a tree.

Finally, if $e$ is any edge of $S$, then $\beta(e)$ lies above $\beta(\mathbf{r})$. So the map $\Omega_p(S) \to \mco_{m,R}$ factors as
\[ \Omega_p(S) \to \Omega_p(T_0) \hookrightarrow \mco_{m,R}. \]
By \cite[2.2.2]{moerdijklecture}, we have a factorization of this first map as a composition of degeneracy maps followed by a composition of face maps
\[ \Omega_p(S) \twoheadrightarrow \Omega_p(T) \hookrightarrow \Omega_p(T_0), \]
from which we get the desired factorization
\[ \Omega_p(S) \twoheadrightarrow \Omega_p(T) \hookrightarrow \mco_{m,R}.\]
\end{proof}

\begin{prop}\label{P:catonopplus}
If a map $(\alpha, \beta): \bi{n}{S} \to \bi{m}{R}$ is in $\dcop^+$, then
\[ d\bi{n}{S} \leq d \bi{m}{R} \] with equality holding if and only if $(\alpha,\beta)$ is an isomorphism.
\end{prop}

\begin{proof}
In the diagram
\begin{equation}
\xymatrix{gen(n,S) \ar@{->}[d]^t \ar@/_3pc/[dd] & gen(m,R) \ar@{->}[d]^t \\
\col(\mco_{n,S}) \ar@{->}[r]^{\hat\beta} & \col(\mco_{m,R}) \\
\col(\mco_{n,S}) \setminus \text{leaves of $S$} %\ar@{->}[r]
\ar@{->}[u] \ar@{->}[r]& \col(\mco_{m,R}) \setminus \text{leaves of $R$} \ar@/_3pc/[uu] \ar@{->}[u]} \label{E:injectiondiagram}
\end{equation}
the rightmost curved arrow takes a color $c$ to the unique generating morphism which has $c$ as its target. The target map $t$ is an injection, as is $\hat\beta$ by Lemma~\ref{L:structureop}\eqref{impliedplusop}.
Since leaves are not the target of any nontrivial morphism, they are never the target for a generating morphism.
In our construction of $\mco_{i,S}$, we saw that every color we added was a target for some nontrivial morphism. Thus every color which is \emph{not} the target of a nontrivial morphism is in $\mco_{0,S} = \Omega_p(S)$.
Thus if $c$ is a non-leaf in $\mco_{n,S}$, there is a nontrivial morphism $h$ with $t(h) = c$. If $\hat\beta(c) = \ell$, then $\hat\beta(h) = \id_{\ell}$, so $h$ has a single input $c'$ and $\hat\beta(c') = \ell$, which cannot happen by injectivity of $\hat\beta$ on colors.  Therefore, we have established the existence of the bottom map in this diagram.

Every map in \eqref{E:injectiondiagram} is an injection, and the curved maps are bijections. It is immediate that $|gen(n,S)| \leq |gen(m,R)|$.
Injectivity of $\alpha$ implies that $\alpha$ has image $m$ whenever $n>0$, so $n \leq m$, establishing the desired inequality.

We now check that if $d\bi{n}{S} = d\bi{m}{R}$, then $\bi{n}{S} = \bi{m}{R}$.
Equality here means that $|gen({n,S})| = |gen({m,R})|$ and $n=m$.
We know $\hat \alpha: \mcc_{n,S} \to \mcc_{m,R}$ is injective on objects by Lemma~\ref{L:structureop}\eqref{impliedplusop}, so gives a bijection $[n]\to[m]$.\footnote{The only exception is possibly when $n=m=0$, in which case the result reduces to that in $\Omega_p$.}
We now know that $gen(n,S) \to gen(m,R)$ is a bijection, and we will show that each edge of $R$ is in the image of $\beta: \Omega_p(S) \to \mco_{m,R}$. If $e \in E(R)$ is not a leaf, then $e=\hat\beta t(h) = t\hat \beta(h)$ for some $h\in gen(n,S)$.
Then
\[ \hat \alpha \mu(h) = \mu(e) =
\begin{cases}
e & e\neq \mathbf{r} \\ 0 & e = \mathbf{r}.
\end{cases} \]
Combining this with the fact that $\alpha$ gives a bijection $[n] \to [m]$, we have $0 \nprec \mu(h)$, so $t(h) \in E(S)$. Thus $e=\hat\beta(t(h))$ is the image of an edge in $S$.

In our construction of $\mco_{i,S}$, we saw that every color we added was a target for some nontrivial morphism. Thus every color which is \emph{not} the target of a nontrivial morphism is in $\mco_{0,S} = \Omega_p(S)$.
Let $e\in E(R)$ be a leaf. Then $e$ is an input for a unique generating morphism $h$ of $\Omega_p(R)$, and $h= \hat \beta(\tilde(h))$ for some unique $\tilde{h}\in gen(n,S)$. Since $0 \prec \mu(h) = \hat \alpha(\mu(\tilde{h}))$, we must have $0 \nprec \mu \tilde{h} $ as well, so $\tilde h\in \mor(\Omega_p(S))$. Letting $\tilde{e}$ be the color in the source of $\tilde{h}$ which maps to $e$, we see that $e = \hat\beta (\tilde e)$ where $\tilde e \in E(S)$.

Thus every edge of $R$ is in the image of $\beta: \Omega_p(S) \to \mco_{m,R}$. If $e\in E(S)$ then $0 \nprec \hat \alpha \mu(e) = \mu \hat \beta(e)$ by injectivity of $\hat \alpha$, so $\beta$ factors through map $\Omega_p(S) \to \Omega_p(R) \to \mco_{m,R}$ where the first is bijective on edges. Returning to the formula from \eqref{E:operaddegree}, we have
\begin{align*}
|\col(\mco_{n,S})| + |V(S)| - |E(S)| & = |gen(n,S)| \\
& = |gen(m,R)| \\
& = |\col(\mco_{m,R})| + |V(R)| - |E(R)|,
\end{align*}
whence
\[ |\col(\mco_{n,S})| + |V(S)| = |\col(\mco_{m,R})| + |V(R)|.\] If we can show that $|\col(\mco_{n,S})|  = |\col(\mco_{m,R})|$, then it will follow that $\Omega_p(S) \to \Omega_p(R)$ is an isomorphism by the corresponding fact in $\Omega_p$ since we will have $|V(S)|=|V(R)|$.
We know that
\begin{multline*}|\col(\mco_{n,S}) \setminus \text{leaves of $S$}| = |gen(n,S)| \\ = |gen(m,R)|  = |\col(\mco_{m,R}) \setminus \text{leaves of $R$} |, \end{multline*}
so we need only to see that there is a bijection between leaves of $S$ and $R$. But we already know that each leaf of $R$ is the image of a leaf of $S$. If a leaf $e$ of $S$ maps to a nonleaf $t(h)$ of $R$, then there is a generating morphism $\tilde h$ with $\hat \beta(\tilde h) = h$, so $t(\tilde h) = e$ by injectivity.
Thus leaves of $S$ map to leaves of $R$, and this map is surjective; it is is injective by Lemma~\ref{L:structureop}\eqref{impliedplusop}.

Thus we have shown that $|\col(\mco_{n,S})| = |\col(\mco_{m,R})|$, and it follows that $|V(S)| = |V(R)|$ so $S=R$.

\end{proof}

\begin{prop}\label{P:catonopminus}
If a map $(\alpha, \beta): \bi{n}{k} \to \bi{m}{\ell}$ is in $\dcop^-$, then
\[ d\bi{n}{k} \geq d \bi{m}{\ell} \] with equality holding if and only if $(\alpha,\beta)$ is an isomorphism.
\end{prop}

\begin{proof}
We first show that each leaf $\ell$ of $R$ is the image of a leaf in $S$. Let $c$ be a minimal element in $\hat \beta^\inv (\ell)$ (under the partial ordering $\prec$; we know this set is nonempty since $\hat\beta$ is surjective on objects). If $c$ is not a leaf, then $c=t(h)$ for some nontrivial morphism $h$, and we see
\[ \ell = \hat \beta (c) = t \hat \beta (h) \]
so $\hat \beta(h) = \id_\ell$ and we have $s(h) \overset{\hat\beta}\mapsto \ell$, contradicting minimality. Thus $c$ must be a leaf.

Furthermore, we know that each leaf of $S$ maps to a leaf of $R$ under $\hat\beta$. We wish to establish a bijection between the leaves of $S$ and the leaves of $R$. Suppose that $\hat \beta (\ell_1) = \hat \beta(\ell_2)$ for two distinct leaves in $S$
and let $h$ be the morphism which is the composition of all vertices in $S$. Then $\hat\beta (h)$ is a morphism which has the same color for two different inputs, which is impossible. Thus $\hat \beta$ induces a bijection of leaves.

Now we observe that
\[ E(S) \cong V(S) \sqcup leaves(S) \]
since all non-leaf edges are the output of a single vertex. Thus we have $|V(S)| - |E(S)| = |V(R)| - |E(R)|$, which we combine with the fact (from the definition of $\dcop^-$) that
\[ |\col(\mco_{n,S})| \geq |\col(\mco_{m,R})| \] to see that
$|gen(n,S)| \geq |gen(m,R)|$ as in
\eqref{E:operaddegree}.

Since $\alpha$ (not $\hat \alpha$) surjects onto the objects of $[m]$ we then have $n\geq m$, which establishes the inequality $d\bi{n}{S} \geq d\bi{m}{R}$.

Suppose that the degrees  $d\bi{n}{S}$ and $d\bi{m}{R}$ are equal, whence $n=m$ and $\hat \beta: \mco_{n,S} \to \mco_{m,R}$ is a bijection on objects.
By Lemma~\ref{L:structureop}\eqref{impliedminusop} we know that every color of $R$ is in the image of $\beta$.
We have $\alpha(0) =0$, so $\beta(\mathbf{r}_S) = \mathbf{r}_R$ since $\mathbf{r}$ is the only object of $\mco_{m,R}$ of moment $0$.
Now if $e$ is any other edge of $S$, then there is a morphism $h$ in $\Omega_p(S)$ with $e\in s(h)$ and $t(h) = \mathbf{r}_S$. Then $\beta(h)$ lies in $\Omega_p(R)$, so $\beta(e)$ must lie in $\Omega_p(R)$ as well. Thus we have $\beta: \Omega_p(S) \to \Omega_p(R)$ a surjection on edges, but $\hat\beta$ was a bijection on edges so this is an isomorphism. Thus we know that $\bi{n}{S} = \bi{m}{R}$.
\end{proof}

Propositions~\ref{P:catonopdecomposition}, \ref{P:catonopplus}, and \ref{P:catonopminus} now imply the following result.

\begin{thm}
The category $\dcop$ is a (strict) Reedy category.
\end{thm}

Let $\Omega_{p,\diamondsuit}$ be the full subcategory of the category of nonsymmetric colored operads with
\[ \ob \Omega_{p,\diamondsuit} = \set{\varnothing} \sqcup \ob \Omega_p, \]
where $\varnothing$ is the operad with empty color set and no morphisms. Define a degree function on this category by
\begin{align*}
d(\varnothing) &= 0\\
d(\Omega_p(S)) &= |V(S)| + 1.
\end{align*}
Noting that $\varnothing$ is initial, we also define wide subcategories
\begin{align*}
\Omega_{p,\diamondsuit}^+ &= \Omega_p^+ \sqcup \setm{\varnothing \to A}{A \in \ob \Omega_{p,\diamondsuit}} \\
\Omega_{p,\diamondsuit}^- &= \Omega_p^- \sqcup \set{\id_\varnothing}.
\end{align*}
We should be explicit that maps in $\Omega_p^-$ are those which are surjective on edges and take leaves to leaves.
It is implicit in \cite{bergermoerdijk} that maps in $\Omega_p^-$ must weakly decrease degree, but surjectivity on edges alone is not enough to guarantee this assumption, as we see in Figure~\ref{F:brokendef}.

\begin{figure}
\includegraphics[width=0.4\textwidth]{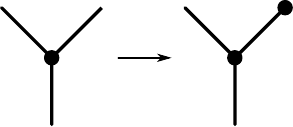}
\caption{A map which is surjective on edges, but increases degree}\label{F:brokendef}
\end{figure}

\begin{lem}\label{L:omegadiamond}
With the degree function and direct and inverse subcategories as above, $\Omega_{p,\diamondsuit}$ is an elegant Reedy category.
\end{lem}

\begin{proof}
The object of $\varnothing$ is the target of a single map, namely the identity on $\varnothing$. Thus decompositions follow as in $\Omega_p$, and $\varnothing \to A$ uniquely decomposes as $\varnothing \overset{-}\to \varnothing \overset{+}\to A$. Compatibility of the direct and inverse categories with the degree function essentially follows from the same fact for $\Omega_p$. Thus $\Omega_{p,\diamondsuit}$ is a Reedy category.

We now turn to elegance. In the pushout constructed in (the planar version of) \cite[2.3.3]{moerdijklecture}, all maps are in $\Omega_p^-$. This pushout is a strong pushout by (the planar version of) \cite[3.1.6]{moerdijklecture}. The only map in $\Omega_{p,\diamondsuit}^-$ involving $\varnothing$ is the identity on $\varnothing$, and
\[ \xymatrix{ \varnothing \ar@{->}[r] \ar@{->}[d] & \varnothing \ar@{->}[d] \\
\varnothing \ar@{->}[r] & \varnothing } \]
is a strong pushout. Elegance follows from \cite[3.4]{bergnerrezk}.
\end{proof}

Recall the partial order on \[ \ob \mcc_{n,S} = \setm{*_e}{e \text{ is a non-root edge of } S} \sqcup \set{0,1,\dots, n},\]
given by  $*_e\prec *_{e'}$ whenever $e$ lies above $e'$ in $S$, $i\prec i'$ whenever $i< i'$, and $*_e \prec i$ for all $i$ and all non-root edges $e$.

\begin{lem}\label{L:precisetargetop}
We have the following:
\begin{enumerate}
\item If $f:(\dots, a, \dots) \to b$ is a morphism of $\mco_{n,S}$ then $\mu(a) \preceq \mu(b)$.\label{orderofobjectsop}
\item If $(\alpha,\beta): \bi{n}{S} \to \bi{m}{R}$ is in $\dcop^-$, then every color in the image of $\beta: \Omega_p(S) \to \mco_{m,R}$ is in $\Omega_p(R)$.\label{precisetargetop}
\end{enumerate}
\end{lem}

\begin{proof}
To prove \eqref{orderofobjectsop}, first notice that if $f$ is an identity the result is immediate. If $f$ is in $\Omega_p(S)$ the statement follows by definition of $\prec$. If $b$ is an object of $\mco_{0,S} = \Omega_p(S)$, then $a$ is an object of $\mco_{0,S}$ since there are no generating morphisms $g$ with $(i,h)\in s(g)$ and $t(g) = b$. W

For the remaining cases, we proceed by induction on the moment of the map $f$. If $b=(i,h)$, there is only one generating morphism with target $b$, $s(h) \to (i,h)$. Thus if $f$ is not an identity,
then $f=\gamma(s(h)\to b; f_1, \dots, f_j)$ for some $f_w$ in $\mco_{i-1,S}$, and $a\in s(f_{w_0})$ for some $w_0$.  By the inductive hypothesis, we have $\mu(a) \preceq \mu(t(f_{w_0})) \preceq i-1 \prec i = \mu(b)$.

\eqref{precisetargetop}: If $e$ is any edge of $S$, there is a morphism $(\dots, e, \dots) \to \mathbf{r}_S$, so we have a morphism $(\dots, \beta(e), \dots) \to \beta(\mathbf{r}_S)$ in $\mco_{m,R}$. But then
\[ \mu(\beta(e)) \preceq \mu(\beta(\mathbf{r}_S)) = \alpha(0) = 0 \]
by \eqref{orderofobjectsop}, Proposition~\ref{P:catopuniversalproperty}, and surjectivity of $\alpha$ onto $[m]$. Since $\mu(\beta(e)) \preceq 0$ for all edges $e$, we have $\beta(e) \in E(R)$ for all $e$.
\end{proof}

\begin{thm}
The Reedy category $\dcop$ is elegant.
\end{thm}

\begin{proof}
Suppose that we have maps $(\alpha_i,\beta_i): \bi{n}{S} \to \bi{m_i}{R_i}$ in $\dcop^-$ for $i=1,2$. By the definition of $\dcop^-$ and the fact that $[n]$ is connected, we may consider $\alpha_i: [n] \to \mcc_{m_i,R_i}$ as a map $\alpha_i: [n] \to [m_i]$, which is in $\Delta^-$. We thus have a strong pushout square
\[ \xymatrix{ [n] \ar@{->}[r]^{\alpha_1} \ar@{->}[d]_{\alpha_2} & [m_1] \ar@{->}[d]^{\delta_1} \\
[m_2] \ar@{->}[r]_{\delta_2} & [w] } \]
since $\Delta$ is elegant. We also consider $\beta_i: \Omega_p(S) \to \mco_{m_i,R_i}$ as a map $\Omega_p(S) \to \Omega_p(R_i)$ which is surjective on colors by Lemma~\ref{L:structureop}\eqref{impliedminusop} and Lemma~\ref{L:precisetargetop}\eqref{precisetargetop}. Since $\Omega_{p,\diamondsuit}$ is an elegant Reedy category by Lemma~\ref{L:omegadiamond}, we have a strong pushout square
\[ \xymatrix{
\Omega_p(S) \ar@{->}[r]^{\beta_1} \ar@{->}[d]_{\beta_2} & \Omega_p(R_1) \ar@{->}[d]^{\gamma_1} \\
\Omega_p(R_2) \ar@{->}[r]_{\gamma_2} & \Omega_p(T)
} \]
in $\Omega_{p,\diamondsuit}$.

We need to show that $(\delta_i, \gamma_i)$ is a morphism in $\dcop^-$, $i=1,2$, and that the corresponding square is a strong pushout. Since $\delta_i(0) = 0$ and $\gamma_i(\mathbf{r}_{R_i}) = \mathbf{r}_T$ by surjectivity, we have that $\mu(\gamma_i(\mathbf{r}_{R_i})) = \mu(\mathbf{r}_T) = 0 = \delta_i(0)$. Thus $(\delta_i,\gamma_i)$ is a map in $\dcop$ by
Proposition~\ref{P:catopuniversalproperty}. It is in $\dcop^-$ by Lemma~\ref{L:structureop}\eqref{impliesminusop}.

It is now left to show that the square
\begin{equation} \begin{gathered} \xymatrix{
F\bi{n}{S} \ar@{->}[r]^{(\alpha_1,\beta_1)} \ar@{->}[d]_{(\alpha_2,\beta_2)} & F\bi{m_1}{R_1} \ar@{->}[d]^{(\delta_1,\gamma_1)} \\
F\bi{m_2}{R_2} \ar@{->}[r]_{(\delta_2,\gamma_2)} & F\bi{w}{T}
} \end{gathered} \label{D:presheafsquareop} \end{equation}
is a pushout square in $\Set^{\dcop^{op}}$. It is enough to show that \[ \xymatrix{
\Hom(\bi{y}{V}, \bi{n}{S}) \ar@{->}[r]^{(\alpha_1,\beta_1)} \ar@{->}[d]_{(\alpha_2,\beta_2)} & \Hom(\bi{y}{V}, \bi{m_1}{R_1}) \ar@{->}[d]^{(\delta_1,\gamma_1)} \\
\Hom(\bi{y}{V}, \bi{m_2}{R_2} )\ar@{->}[r]_{(\delta_2,\gamma_2)} & \Hom(\bi{y}{V}, \bi{w}{T})
} \]
is a pushout diagram in $\Set$ for each object $\bi{y}{V}$ in $\dcop$.

We have
\begin{equation}
\begin{aligned}
&\Hom(\bi{y}{V}, \bi{a}{B}) \\ &= \setm{\sigma \times \tau}{\sigma: [y]\to [a], \tau: \Omega_p(V) \to \Omega_p(B), \sigma(0) =\mu\tau(\mathbf{r}_V) }  \\
&\subseteq  \Hom([y], [a]) \times \Hom(\Omega_p(V),\Omega_p(B))
\end{aligned} \label{E:splituphomthingop}
\end{equation}
and we compute that the pushout should be
\begin{equation}
\left[ \Hom(\bi{y}{V}, \bi{m_1}{R_1}) \amalg \Hom(\bi{y}{V}, \bi{m_2}{R_2} ) \right] / \sim \label{E:actualpushoutop}
\end{equation}
where
\[  (\sigma_1,\tau_1) \sim (\sigma_2, \tau_2) \text{ when } \sigma_1 \alpha_1 = \sigma_2 \alpha_2 \text{ and } \tau_1 \beta_1 = \tau_2 \beta_2.\]
However, this pushout is contained in
\begin{equation} \left( \coprod_{i=1,2} \Hom([y],[m_i]) \times \Hom(\Omega_p(V),\Omega_p(R_i)) \right) / \sim \label{E:stretchedop}\end{equation}
where $\sigma_1\times \tau_1 \sim \sigma_2\times \tau_2$ when $\sigma_1 \alpha_1 = \sigma_2 \alpha_2$  and  $\tau_1 \beta_1 = \tau_2 \beta_2,$
with the extra conditions being that $\sigma_i(0) = \mu\tau(\mathbf{r}_{V})$.
We see that \eqref{E:stretchedop} is equal to
\[ \Hom([y],[w]) \times \Hom(\Omega_p(V), \Omega_p(T)) \]
and by \eqref{E:splituphomthingop} we have that
\[ \Hom(\bi{y}{V}, \bi{w}{T}) \ci \Hom([y],[w]) \times \Hom(\Omega_p(V),\Omega_p(T)) \]
is equal to \eqref{E:actualpushoutop}. Thus when we evaluate the diagram of presheaves \eqref{D:presheafsquareop} on any object of $\dcop$ we get a pushout, so \eqref{D:presheafsquareop} is itself a pushout. Hence $\dcop$ is elegant.
\end{proof}

\section{Symmetric operads and nonplanar trees}\label{S:symmetric}

In this section we extend the category $\dcop$ to a category $\dcos$, which controls rooted actions of categories on symmetric operads. Of key importance is the adjunction $\Sigma \colon \nsoperads \rightleftarrows \operads \colon U$ between nonsymmetric operads and symmetric operads. The left adjoint $\Sigma$ is the \emph{symmetrization} functor where $\Sigma\mco$ has the same set of colors as $\mco$, and $(\Sigma \mco) (c_1, \dots, c_n; c) = \coprod_{\sigma \in \Sigma_n} \mco(c_{\sigma(1)}, \dots, c_{\sigma(n)}; c)$.
We can use $\Sigma$ to describe the left-adjoint to the forgetful functor $U:\opact \to \opactns$.
Consider a rooted action $\mcc \circlearrowright \mco$ of a category $\mcc$ on a nonsymmetric operad $\mco$, and suppose that $f$ is a morphism of $\mcc$ and $g$ is in $\mco$ so that $f\bullet g$ is defined. Let $\sigma$ be a permutation, and define (as required by \eqref{E:symaction}) $f\bullet (\sigma^*g) := \sigma^*(f\bullet g)$. This data gives a rooted action $\mcc \circlearrowright \Sigma\mco$, and we call this assignment $\Sigma: \opactns \to \opact$, which one can check is left-adjoint to the forgetful functor $U$.
We obtain the following result from adjointness and Proposition~\ref{P:catopuniversalproperty}.

\begin{prop}
Suppose that there is a rooted action of a category $\mca$ on a symmetric operad $\mcp$. Then a map
\[ X \colon \Sigma\bi{n}{S} \to \mca \betweener \mcp \]
in $\opact$ is equivalent to a pair of morphisms
\begin{align*}
\alpha: [n] &\to \mca &
\beta: \Omega(S) &\to \mcp
\end{align*}
satisfying $\mu(\beta(\mathbf{r})) = \alpha(0)$. 
\end{prop}

We thus define $\dcos$ as the full subcategory of $\opact$ whose objects are $\Sigma\bi{n}{S}$.
As in Section~\ref{S:catoncat}, we can define a nerve functor
\begin{align*}
 \nerve\colon \opact &\to \Set^{\dcos^{op}}\\
\nerve(\mca \circlearrowright \mcp)_{\Sigma\bi{n}{S}} &= \Hom_{\opact} (\Sigma\bi{n}{S},  \mca \circlearrowright \mcp).
\end{align*}

\begin{prop}
The functor $\nerve\colon \opact\to \Set^{\dcos^{op}}$ is fully faithful.
\end{prop}
\begin{proof}
The proof is a slight modification of Proposition~\ref{P:ffnerve}, using the fact that the evident \emph{dendroidal nerve}
\[ \operads \to \Set^{\Omega^{op}} \] is fully faithful \cite{mw}.
\end{proof}

We could use the methods from Section~\ref{S:catoperad} to show that $\dcos$ is a generalized Reedy category, but it is more efficient to utilize the notion of a crossed group as described in \cite[\S 2]{bergermoerdijk}. A \emph{crossed group} $G$ on a  small category $\mathcal{R}$ is a functor $\mathcal{R}^{op} \to \Set$ together with, for each object $r$ of $\mathcal{R}$, a group structure on $G_r$ and left $G_r$-actions on the hom-sets $\Hom_\mathcal{R}(s,r)$ satisfying certain compatibility conditions.
For any small category $\mathcal{R}$ and crossed $\mathcal{R}$-group $G$, the \emph{total category} $\mathcal{R} G$ is the category with the same objects as $\mathcal{R}$, and with morphisms $r\to s$ the pairs $(\alpha,g)$ where $\alpha:r\to s$ belongs to $\mathcal{R}$, and $g\in G_r$. Composition of $(\alpha,g):s\to t$ and $(\beta,h):r\to s$ is defined as
$(\alpha,g)\circ(\beta,h)=(\alpha\cdot g_*(\beta),\beta^*(g)\cdot h).$ Finally, if $\mathcal{R}$ is a generalized Reedy category, we say that $G$ is compatible with the generalized Reedy structure if
\begin{enumerate}
\item the $G$-action respects $\mathcal{R}^+$ and $\mathcal{R}^-$ (i.e. if $\alpha:r\to s$ belongs to $\mathcal{R}^\pm$ and $g\in G_s$ then $g_*(\alpha):r\to s$ belongs to $\mathcal{R}^\pm$); and
\item if $\alpha:r\to s$ belongs to $\mathcal{R}^-$ and $g\in G_s$ is such that $\alpha^*(g)=e_r$ and $g_*(\alpha)=\alpha$, then $g=e_s$.
\end{enumerate}

As a key example, there is a crossed group $G$ on $\Omega_p$ so that the total category $\Omega_pG$ is equivalent to $\Omega$. Let $G$ be this crossed group on $\Omega_p$ as in \cite[2.8]{bergermoerdijk}.
We will use this crossed group in what follows, so, as a technical point, we take $\Omega$ to have objects the planar trees, so that $\Omega_p$ is a wide subcategory of $\Omega$ and $\Omega=\Omega_pG$.%

Suppose that we have a morphism $(\alpha, \beta):\Sigma\bi{n}{S} \to \Sigma\bi{m}{R}$ of $\dcos$. The morphism $\beta: \Omega(S) \to \Sigma \mco_{m,R}$ decomposes as
\[ \Omega(S) \twoheadrightarrow \Omega(T) \overset{\Sigma f}{\hookrightarrow} \Sigma \mco_{m,R} \] using the argument of Lemma~\ref{L:opfactor}, where $f: \Omega_p(T) \to \mco_{m,R}$ is a map of nonsymmetric operads and $\Omega(S) \twoheadrightarrow \Omega(T)$ is a map in $\Omega^-$; this decomposition is unique.
Furthermore, there is a unique factorization \[ \Omega(S) \overset\cong\to \Omega(S) \overset{\Sigma g}{\twoheadrightarrow} \Omega(T)\] as in \cite[\S 2.3.2]{moerdijklecture}, where $g: \Omega_p(S) \to \Omega_p(T)$ is a planar map and $\Omega(S) \to \Omega(S)$ is in $G_S$.
This decomposition of $\beta$ gives a unique decomposition
\[ \Sigma\bi{n}{S} \overset\cong\to \Sigma\bi{n}{S} \overset{\Sigma h}\longrightarrow \Sigma\bi{m}{R} \]
where $h: \bi{n}{S} \to \bi{m}{R}$ is in $\dcop$ and the first map comes from the action of $G_S$.

\begin{thm}
The category $\dcos$ admits the structure of a generalized Reedy category extending the Reedy structure on $\dcop$.
\end{thm}
\begin{proof}
We just indicated a unique factorization of morphisms in $\dcos$, which shows that $\dcos$ is the total category of the crossed group $G$ on $\dcop$ by \cite[2.5]{bergermoerdijk}.
Moreover, this crossed group is compatible with the Reedy structure on $\dcop$, so $\dcos$ inherits a generalized Reedy structure extending that on $\dcop$ by \cite[2.10]{bergermoerdijk}.
\end{proof}

\end{document}